\documentclass[review,onefignum,onetabnum]{siamart190516}

\usepackage[utf8]{inputenc}
\usepackage{amsmath}
\usepackage{amssymb}
\usepackage{mathtools}

\usepackage{tikz, pgfplots}
\usetikzlibrary{tikzmark,calc}
\usetikzlibrary{decorations.pathmorphing}
\pgfplotsset{width=7cm, compat=1.10}
\usepgfplotslibrary{fillbetween}
\usepgfplotslibrary{external} 
\usepackage{bm}
\usepackage{soul}

\usepackage{hyperref}
\usepackage{booktabs}

\usepackage[square,sort,comma,numbers]{natbib}

\usepackage{lipsum}
\usepackage{amsfonts}
\usepackage{graphicx}
\usepackage{epstopdf}

\usepackage{algorithm2e}

\ifpdf
  \DeclareGraphicsExtensions{.eps,.pdf,.png,.jpg}
\else
  \DeclareGraphicsExtensions{.eps}
\fi

\newsiamremark{remark}{Remark}
\newsiamremark{hypothesis}{Hypothesis}
\crefname{hypothesis}{Hypothesis}{Hypotheses}
\newsiamthm{claim}{Claim}

\headers{Multilevel Stochastic Gradient for Optimal Control}{M. Martin, F. Nobile, and P. Tsilifis}

\title{A multilevel Stochastic Gradient method for {PDE}-constrained Optimal Control Problems with uncertain parameters\thanks{Submitted to SIAM Journal on Optimization, December 26, 2019.\funding{F. Nobile, P. Tsilifis, and M. Martin have received support from the Center for ADvanced MOdeling Science (CADMOS).}}}

\author{Matthieu Martin\thanks{Criteo, Parc Sud Galaxie, 38130 Echirolles, Grenoble, France
  (\email{matthieu.martin@epfl.ch})}
\and Fabio Nobile\thanks{CSQI, Institute of Mathematics,
  {\'E}cole Polytechnique F{\'e}d{\'e}rale de Lausanne, 1015 Lausanne,
  Switzerland 
  (\email{fabio.nobile@epfl.ch}, \email{panagiotis.tsilifis@epfl.ch}, \url{https://www.epfl.ch/labs/csqi/}).}
\and Panagiotis Tsilifis\footnotemark[3]
  }

\usepackage{amsopn}

\newtheorem{assumption}{Assumption}

\newcommand{\R}{\mathbb{R}}
\newcommand{\aaa}{\lambda}
\newcommand{\bbb}{\mu}
\newcommand{\aaab}{{\lambda}}
\newcommand{\bbbb}{{\mu}}
\newcommand{\E}{\mathbb{E}}
\newcommand{\V}{\mathbb{V}ar}

\newcommand\numberthis{\addtocounter{equation}{1}\tag{\theequation}}
\newcommand{\uu}{u^*}
\newcommand{\Erg}{E_{L, \overrightarrow{p}}^{\textrm{RMLMC}}}
\newcommand{\Er}{E_{L_j, \overrightarrow{p}^j}^{\textrm{RMLMC}}}
\newcommand{\Em}{E_{L_j, \overrightarrow{N}_j}^{\textrm{MLMC}}}
\newcommand{\Emi}{E_{L_i, \overrightarrow{N}_i}^{\textrm{MLMC}}}
\newcommand{\Emg}{E_{L, \overrightarrow{N}}^{\textrm{MLMC}}}
\newcommand{\Emeps}{E_{L(\epsilon), \overrightarrow{N}(\epsilon)}^{\textrm{MLMC}}}

\newcommand{\uuh}{u^{h*}}

\newcommand{\Sad}{\mathcal{S}_{ad}}

\DeclareMathOperator*{\argmin}{arg\,min}
\DeclareMathOperator*{\dive}{div}

\usepackage{xcolor}
\definecolor{mycolor1}{rgb}{0.00000,0.44700,0.74100}%
\definecolor{mycolor2}{rgb}{0.85000,0.32500,0.09800}%
\definecolor{darkgreen}{RGB}{0,102,34}
\definecolor{darkblue}{RGB}{57,61,125}

\begin{document}

\maketitle

\begin{abstract}
In this paper, we present a multilevel Monte Carlo (MLMC) version of the Stochastic Gradient (SG) method for optimization under uncertainty, in order to tackle Optimal Control Problems (OCP) where the constraints are described in the form of PDEs with random parameters. The deterministic control acts as a distributed forcing term in the random PDE and the objective function is an expected quadratic loss. We use a Stochastic Gradient approach to compute the optimal control, where the steepest descent direction of the expected loss, at each iteration, is replaced by independent MLMC estimators with increasing accuracy and computational cost. The refinement strategy is chosen a-priori such that the bias and, possibly, the variance of the MLMC estimator decays as a function of the iteration counter. Detailed convergence and complexity analyses of the proposed strategy are presented and asymptotically optimal decay rates are identified such that the total computational work is minimized. We also present and analyze an alternative version of the multilevel SG algorithm that uses a randomized MLMC estimator at each iteration. Our methodology is validated through a numerical example.

\end{abstract}

%


\section{Introduction}


Over the past few years, the use of stochastic optimization algorithms in various scientific computing areas has maintained an increasing momentum \cite{johnson_zhang, zinkevich, konecny} and, particularly, gradient-descent-type algorithms have been routinely used to solve robust optimization problems under uncertainty, that typically involve the optimization of an expected loss function. In their general formulation, these problems are stated as
$$\uu \in \argmin_{u \in U} J(u),$$ 
where the utility or loss functional $J(u)$ involves an expectation, that is  $$J(u)=\E_{\omega}[f(u, \omega)]$$
where $\omega \in \Gamma$ denotes a random elementary event and $f$ a random "loss" or objective functional. The classic Robbins-Monro Stochastic Gradient (SG) algorithm \cite{robbins1951} then writes:
$$u_{j+1}=u_{j}-\tau_j \nabla f(u_j, \omega_j)$$ where the sequence of $ \omega_j$ is independent and identically distributed (iid) and the step-size is usually chosen as $\tau_j=\tau_0/j$, with $\tau_0$ sufficiently large. In particular, when $u \mapsto f(u, \omega)$ is strongly convex for a.e. $\omega \in \Gamma$, with bounded or Lipschitz gradient, the SG algorithm guarantees an algebraic convergence rate on the mean squared error, i.e. $\E[\| u_j-\uu \|^2] \lesssim 1/j$, where $\uu$ denotes the exact optimal control.
In our setting, the evaluation of the loss functional $f(u_j, \cdot)$ involves the solution of a PDE, which can not be done exactly, except in very simple cases and requires a discretization step. This, in turn, induces an error in the evaluation of the loss functional, which can be kept small at the price of a high computational cost.

The multilevel Monte Carlo method has been proven to be very effective in reducing the cost of computing expectations of output functionals of differential models, compared to classic Monte Carlo estimators, by exploiting a hierarchy of discretizations, with increasing accuracy and computational costs. Essentially, the MLMC estimator performs the main bulk of the required computational work on coarse (and cheap) discretizations while a small number of evaluations performed on finer discretizations characterize the correction terms.

The multilevel paradigm has been introduced by Heinrich \cite{Heinrich1998} for parametric integration. It has been extended to weak approximations of stochastic differential equations (SDEs) in \cite{giles:MLMC} and has shown its efficiency, as a tool in numerical computations. 
The applicability of MLMC methods to uncertainty quantifiction problems, involving PDEs with random parameters, has been recently demonstrated along with extensive mathematical analysis \cite{Barth2011,Barth2013,Charrier2013,cliffe2011,Mishra2012,Teckentrup2013,nobile.tesei:multi}. 

In the most favorable cases, it has been shown that the cost $W$ of computing the expected value of some output quantity of a stochastic differential model with accuracy $tol$, scales as $W \lesssim tol^{-2}$, and does not see the cost of solving the problem on fine discretizations.
In this work, we use MLMC within a SG algorithm, by replacing the single realization $\nabla f (u_j, \omega_j)$ by a MLMC estimator, that builds on a hierarchy of finite element (FE) discretizations. In particular, we present a full convergence and complexity analysis of the resulting MLSG algorithm in the case of a quadratic, strongly convex, OCP. By reducing the bias and the variance of the MLMC estimator as a function of the iterations at a proper rate, we are able to recover in some cases an optimal complexity $W \lesssim tol^{-2}$ in the computation of the optimal control, analogous to the one for the computation of a single expectation. This result considerably improves the one in our previous work \cite{MartinKrumscheidNobile} where we have studied the SG method in which a single realization $\nabla f (u_j, \omega_j)$ is taken at each iteration, and computed on progressively finer discretization over the iterations. An alternative way to combine MLMC and Stochastic Gradient, although not in the context of PDE constrained OCPs, has been proposed in \cite{Frikha2016} and also leads to the optimal complexity $W \lesssim tol^{-2}$ in favorable cases. The key idea there is to build a sequence of coupled standard Robbins-Monro algorithms with different discretization levels in order
to construct a multilevel estimator of the optimal control. In this work we propose, instead, to construct a single SG algorithm that uses at each iteration a multilevel estimate for $\E[\nabla f(u, \cdot)]$, with increasing accuracy over the iterations. 

Our approach is similar to the one proposed and analyzed in \cite{ML} although it is particularized to the specific infinite-dimensional PDE-constrained OCP addressed in this work. Our analysis technique differs from the one in \cite{ML}, is applicable to infinite dimensional strongly convex optimization problems and leads to slightly different choices of bias and variance decay rates than those in \cite{ML}, although the final complexity results are comparable.

Furthermore, we also consider a randomized version of the MLSG algorithm, which uses the unbiased multilevel Monte Carlo algorithm proposed in  \cite{RheeGlynn2012, RheeGlynn2015} (see also \cite[Section 2.2]{Giles}). In this randomized version, named RMLSG, the full MLMC estimator is replaced by an estimator that consists of only one correction term that is evaluated on a single sample and is selected randomly, after assigning a probability mass function on the levels $\{\l_j\}_{j\geq 0}$. We show that RMLSG also achieves optimal complexity $W \lesssim tol^{-2}$ in some cases (the same as for MLSG). The main advantage of this randomized version, w.r.t. the one that uses a full MLMC estimator at each iteration, is that it requires less parameters to tune, which is preferable, from a numerical point of view.

We mention also two more recent works proposing MLMC methods for PDE constrained optimal control problems. The work \cite{ali_ullman} addresses the rather different situation of a pathwise optimization, where a different optimal control is computed for each realization of the random coefficients in the PDE by a standard deterministic optimization procedure. The multilevel Monte Carlo estimator is then used to compute the expectation of the (random) optimal control.
The work \cite{vanbarel} considers a very similar quadratic PDE constrained OCP as in this work, with an additional penalty term on the variance of the state. On the other hand, the MLMC estimator is not proposed withing a Stochastic Gradient framework, rather following a Sample Average Approximation approach (see \cite{ShapiroStochasticProgramming}) in which the MLMC estimator is used to approximate upfront the expected loss function and a (deterministic) nonlinear conjugate gradient method is then used to compute the optimal control of the problem thus discretized.


The outline of this paper is as follows. In Section \ref{sec2}, we present the problem setting, and recall some results on the existence and uniqueness of solutions to the optimal control problem, and its finite element approximation. Next, we review the main ingredients of multilevel Monte Carlo estimation. In Section \ref{MLMC-SG}, we introduce the multilevel Monte Carlo Stochastic Gradient method. Our algorithm uses a hierarchy of uniformly refined meshes where the mesh sizes form a geometric sequence. This is justified from the arguments in \cite{HajiNobileSchwerinTempone}. We discuss the optimal strategy to reduce the bias and variance of the MLMC estimator over the iterations and derive a complexity result.
In Section \ref{randomized} we present the randomized version of the MLSG algorithm and derive the corresponding complexity result.
Section \ref{numer} presents a numerical example, where the theoretical results from Sections \ref{MLMC-SG} and \ref{randomized} are validated. Finally, Section \ref{ccl} summarizes our conclusions and future perspectives of this work.


\section{Problem setting}\label{sec2}

We start by introducing the primal problem that will be part of the Optimal Control Problem (OCP) discussed next. We consider the problem of finding the solution $y: D \times \Gamma \rightarrow \R$ of the elliptic random PDE
\begin{equation}
\left\{
\begin{array}{rcll}
-\dive(a(x,\omega)\nabla y(x,\omega)) & = & g(x)+u(x), \hspace{10mm} & x \in D, \hspace{3mm} \omega \in \Gamma, \\
y(x,\omega) & = & 0,  \hspace{10mm} & x \in \partial D, \hspace{3mm} \omega \in \Gamma, 
\end{array}
\right.
\label{eq:primal}
\end{equation}
where $D \subset \R^d$ denotes the physical domain, assumed here to be open and bounded, and $(\Gamma, \mathcal{F}, \mathbb{P})$ is a complete probability space. The diffusion coefficient $a$ is a random field, $g$ is a deterministic source term and $u$ is the deterministic control. The solution of \eqref{eq:primal} for a given control $u$ will be equivalently denoted $y_{\omega}(u)$, or simply $y(u)$ in what follows. Let $U= L^2(D)$  be the set of all admissible control functions and $Y=H_0^1(D)$ the space of the solutions of \eqref{eq:primal} endowed with the norm $\|v\|_{H_0^1(D)}=\| \nabla v \|$ where $\|\cdot \|$ denotes the $L^2(D)$-norm induced by the inner product $\langle \cdot,\cdot \rangle$. The ultimate goal is to determine an optimal control $u^*$, in the sense that:
\begin{equation}
\uu \in \argmin_{u \in U} J(u), \quad \mathrm{s.t.} \quad y_{\omega}(u) \in Y \quad \mathrm{solves} \quad \eqref{eq:primal} \quad \mathrm{almost} \hspace{1mm} \mathrm{surely} \hspace{1mm} \mathrm{(a.s.)} \hspace{1mm} \mathrm{in} \hspace{1mm} \Gamma.
\label{eq:ocp}
\end{equation}
Here, $J(u):=\E[f(u, \cdot)]$ is the objective function with $f(u, \omega)=\frac{1}{2}\|y_{\omega}(u) - z_d\|^2 +\frac{\beta}{2}\|u\|^2$ and $z_d$ is the target function that we would like the state $y$ to approach as close as possible, in a mean squares sense. We use assumptions and results from \cite[Sections 2]{MartinKrumscheidNobile} to guarantee well posedness of \eqref{eq:ocp}.

\begin{assumption}
\label{as:4}
The diffusion coefficient $a \in L^{\infty}(D \times \Gamma)$ is bounded and bounded away from zero a.e. in $D \times \Gamma$, i.e.
\[
\exists \quad a_{\mathrm{min}}, a_{\mathrm{max}} \in \R \hspace{4mm} \mathrm{such} \hspace{1mm} \mathrm{that} \hspace{4mm} 0<a_{\mathrm{min}}\leq a(x,\omega) \leq  a_{\mathrm{max}} \hspace{4mm} \mathrm{a.e.} \hspace{1mm} \mathrm{in} \hspace{1mm} D \times \Gamma.
\]
\end{assumption}

\begin{assumption}
\label{as:1}
The regularization parameter $\beta$ is strictly positive, i.e.~$\beta > 0$, and the deterministic source term is such that $g \in L^2(D)$.
\end{assumption}

In what follows, we denote with $\nabla J(u)$ the $L^2(D)$-functional representation of the Gateaux derivative of $J$, which is defined as
\[
\int_D \nabla J(u) \delta u \hspace{1mm}\mathrm{d}x = \lim_{\epsilon \rightarrow 0} \frac{J(u+\epsilon \delta u)-J(u)}{\epsilon}, \quad \forall \delta u \in L^2(D).
\]
Then existence and uniqueness of the OCP \eqref{eq:ocp} can be stated as follows. 
\begin{theorem}
\label{theoremass}
\label{exist}
Under Assumptions \ref{as:4} and \ref{as:1},  the OCP \eqref{eq:ocp} admits a unique optimal control $\uu \in U$. Moreover
\begin{equation}\label{gradi}
\nabla J(u)=\E[\nabla f(u,\cdot)]\mathrm{\hspace{2mm} with \hspace{2mm}} \nabla f(u,\omega)=\beta u+p_{\omega}(u),
\end{equation}
%
where $p_{\omega}(u)=p$ is the solution of the adjoint problem (a.s. in $\Gamma$)
\begin{equation}
\left\{
\begin{array}{rcll}
-\dive(a(\cdot,\omega)\nabla p(\cdot,\omega)) & = & y(\cdot,\omega)-z_d \hspace{10mm} & \mathrm{in}\hspace{1mm} D, \\
p(\cdot,\omega) & = & 0  \hspace{10mm} &  \mathrm{on}\hspace{1mm} \partial D. 
\end{array}
\right.
\label{eq:dual}
\end{equation}
\end{theorem}
We recall also the weak formulation of \eqref{eq:primal}, which reads
\begin{equation}
\label{primmal}
\mathrm{find} \hspace{1mm} y_{\omega} \in Y \hspace{1mm} s.t. \hspace{1mm} b_{\omega}(y_{\omega},v)=\langle g+u,v \rangle  \quad \forall v \in Y \quad \quad \mathrm{for \hspace{1mm} a.e.} \hspace{1mm} \omega \in \Gamma,
\end{equation} 
where  $b_{\omega}(y,v):= \int_D a(x,\omega) \nabla y(x) \nabla v(x) dx$.
Similarly, the weak form of the adjoint problem \eqref{eq:dual} reads:
\begin{equation}
\label{dual}
\mathrm{find} \hspace{1mm} p_{\omega} \in Y \hspace{1mm} s.t. \hspace{1mm}b_{\omega}(v,p_{\omega})=\langle v,y_{\omega}-z_d \rangle  \quad \forall v \in Y \quad \quad \mathrm{for \hspace{1mm} a.e.} \hspace{1mm} \omega \in \Gamma.
\end{equation}
We can thus rewrite the OCP \eqref{eq:ocp} equivalently as:
\begin{equation} \label{eqn:mother1}
\left\{ 
\arraycolsep=1.0pt\def\arraystretch{1}
\begin{array}{l}
\min_{u \in  U}J(u) :=\frac{1}{2}\E[\| y_{\omega}(u)-z_d \|^2]+\frac{\beta}{2}\| u  \|^2 \\
\mathrm{s.t.} \quad y_{\omega}(u) \in Y \quad \mathrm{solves}\\
b_{\omega}(y_{\omega}(u),v)=\langle g+u,v \rangle  \quad \forall v \in Y \quad \quad \mathrm{for \hspace{1mm} a.e.} \hspace{1mm} \omega \in \Gamma.
\end{array}
\right.
\end{equation}
Following \cite{MartinKrumscheidNobile},  we now recall two regularity results about Lipschitz continuity and strong convexity of $f$ in the particular setting of the problem considered here.
\begin{lemma}[Lipschitz continuity]
\label{lemma:Lip}
The random functional $f$ is such that:
\begin{equation}
\label{IL}
\|\nabla f(u,\omega)-\nabla f(v,\omega)\| \leq Lip \| u-v \| \quad \forall u,v \in U \mathrm{\hspace{1mm}and \hspace{1mm} a.e.\hspace{1mm}} \omega \in \Gamma,
\end{equation}
with $Lip=\beta+\frac{C_p^4}{a_{min}^2}$, where $C_p$ is the Poincar\'e constant, $C_p=\sup_{v \in Y\setminus\{0\}} \frac{\| v \|}{\| \nabla v \|}$.
\end{lemma}

\begin{lemma}[Strong convexity] \label{lemma:SC}
The random functional $f$ is such that:
\begin{equation}
\label{ISC}
\frac{l}{2} \| u-v \|^2 \leq \langle \nabla f(u,\omega)-\nabla f(v,\omega) , u-v \rangle \quad \forall u,v \in U \mathrm{\hspace{1mm}and \hspace{1mm} a.e.\hspace{1mm}} \omega \in \Gamma,
\end{equation}
with $l=2 \beta$.
\end{lemma}


\subsection{Finite Element approximation}
In order to compute numerically an optimal control we consider a Finite Element (FE) approximation of the infinite dimensional OCP \eqref{eqn:mother1}.
Let us denote by $\{ \tau_h \}_{h>0}$ a family of regular triangulations of $D$ and  choose $Y^h$ to be the space of continuous piece-wise polynomial functions of degree at most $r$ over $\tau_h$ that vanish on $\partial D$, i.e. $Y^h=\{ y \in C^0(\overline{D}): y\vert_K \in \mathbb{P}_r(K),\quad \forall K \in \tau_h, y \vert_{\partial D}=0 \} \subset Y$, and $U^h=Y^h$. We reformulate the OCP \eqref{eqn:mother1} as a finite dimensional OCP in the FE space:

\begin{equation}\label{eqn:discr2}
\left\{ 
\arraycolsep=1pt\def\arraystretch{1}
\begin{array}{l}
\min_{u^h \in  U^{h}}J^h(u^h) :=\frac{1}{2}\E[\| y^h_{\omega}(u^h)-{z_d} \|^2]+\frac{\beta}{2}\| u^h  \|^2 \\
\mathrm{s.t.\hspace{1mm}} y_{\omega}^h \in Y^h \mathrm{\hspace{1mm} and}\\
b_{\omega}(y^h_{\omega}(u^h),v^h)=\langle u^h+g,v^h \rangle \quad \forall v^h \in Y^h \quad \mathrm{for \hspace{1mm} a.e. \hspace{1mm}} \omega \in \Gamma.
\end{array}
\right.
\end{equation}
Under the following regularity assumption on the domain and diffusion coefficient:
\begin{assumption}
\label{as:2}
The domain $D \subset \R^d$ is polygonal convex and the random field $a\in L^{\infty}(D \times \Gamma)$ is such that $\nabla a \in L^{\infty}(D \times \Gamma)$,
\end{assumption} 
\noindent the following error estimate has been obtained in \cite{MartinKrumscheidNobile}. In order to lighten the notation, we omit the subscript $\omega$ in $y_{\omega}(\cdot)$ and $p_{\omega}(\cdot)$ from now on.
\begin{theorem}
\label{c1}
Let $\uu$ be the solution to the optimal control problem \eqref{eqn:mother1}, and denote by $\uuh$ the solution to its finite element counterpart \eqref{eqn:discr2}. Suppose that $y(\uu), p(\uu) \in L_{\mathbb{P}}^2(\Gamma; H^{r+1}(D))$ and Assumption \ref{as:2} holds; then
\begin{multline}
\label{febound}
\| \uu-\uuh \|^2+  \E[\| y(\uu)-y^h(\uuh) \|^2]+h^2 \E[\| y(\uu)-y^h(\uuh) \|_{H^1_0}^2]\\ 
\leq A_1 h^{2r+2} \left(\E[|y(\uu)|^2_{H^{r+1}}]+\E[| p(\uu) |^2_{H^{r+1}}]\right),
\end{multline}
with a constant $A_1$ independent of $h$.
\end{theorem}
Notice that the OCP \eqref{eqn:discr2} can be equivalently formulated in $U$ instead of $U^h$:
\begin{equation}\label{eqn:discr3}
\left\{ 
\arraycolsep=1pt\def\arraystretch{1}
\begin{array}{l}
\min_{u \in  U}J^h(u) :=\frac{1}{2}\E[\| y^h_{\omega}(u)-{z_d} \|^2]+\frac{\beta}{2}\| u  \|^2 \\
\mathrm{s.t.\hspace{1mm}} y_{\omega}^h \in Y^h \mathrm{\hspace{1mm} and}\\
b_{\omega}(y^h_{\omega}(u),v^h)=\langle u+g,v^h \rangle \quad \forall v^h \in Y^h, \quad \mathrm{for \hspace{1mm} a.e. \hspace{1mm}} \omega \in \Gamma.
\end{array}
\right.
\end{equation}
Indeed, if we decompose any $u \in U$ into $u=u_h+w$ with $u_h \in U^h$ and $\langle w,v^h \rangle=0, \quad \forall v^h \in U^h$, it follows that 
\begin{equation}
J^h(u)=\frac{1}{2}\E[\| y^h_{\omega}(u^h)-{z_d} \|^2]+\frac{\beta}{2}\| u^h  \|^2+\frac{\beta}{2}\| w  \|^2,
\end{equation}
so clearly the optimal control $\widetilde{u}^*$ of \eqref{eqn:discr3} satisfies $\widetilde{u}^*=u^{h*}+w^*$ with $w^*=0$ and $u^{h*}$ solution of \eqref{eqn:discr2}, i.e. the optimal control of \eqref{eqn:discr3} is indeed a FE function in $U^h$. For later developments it will be more convenient to consider the formulation \eqref{eqn:discr3} rather than \eqref{eqn:discr2}.\\
Following analogous developments as in \cite{MartinKrumscheidNobile}, it is straightforward to show that $\forall u \in U$
\begin{equation} \label{FEgradient}
\nabla J^h(u)=\E[\nabla f^h(u,\cdot)]\mathrm{\hspace{2mm} with \hspace{2mm}} \nabla f^h(u,\omega)=\beta u+p_{\omega}^h(u) \in U,
\end{equation}
where $p_{\omega}^h(u)$ solves the FE adjoint problem which reads
\begin{equation}
\mathrm{find} \hspace{1mm} p_{\omega}^h(u) \in Y^h \hspace{1mm} s.t. \hspace{1mm}b_{\omega}(v^h,p^h_{\omega}(u))=\langle v^h,y_{\omega}^h(u)-z_d \rangle  \quad \forall v^h \in Y^h, \quad  \mathrm{for \hspace{1mm} a.e.} \hspace{1mm} \omega \in \Gamma,
\end{equation}
and $y_{\omega}^h(u)$ solves the primal problem formulated in \eqref{eqn:discr3}. Moreover, $\nabla f^h(u,\cdot)$ satisfies Lipschitz and convexity properties analogous to those in Lemmas \ref{lemma:Lip} and \ref{lemma:SC}, with the same constants: 
\begin{lemma}[Lipschitz continuity and strong convexity]\label{lemma:Lipconv_h}
For any triangulation $\tau_h$, $h\geq 0$, the random functional $f^h$ satisfies
$$\| \nabla f ^{h}(u,\omega)-\nabla f ^{h}(v,\omega) \| \leq Lip \| u-v \| \quad \forall u,v \in U, \quad\mathrm{for\hspace{1mm} a.e. \hspace{1mm}} \omega \in \Gamma.$$
and 
$$ \frac{l}{2} \| u-v \|^2 \leq \langle u-v,  \nabla f ^{h}(u,\omega)-\nabla f ^{h}(v,\omega) \rangle  \quad \forall u,v \in U, \quad \mathrm{for\hspace{1mm} a.e. \hspace{1mm}} \omega \in \Gamma$$
with constants $Lip$ and $l$ as in Lemmas \ref{lemma:Lip} and \ref{lemma:SC} that are independent of $h$.
\end{lemma}
We conclude this section by stating an error bound on the FE approximation of the functional $f$ when evaluated at the optimal control  $\uu$.
\begin{lemma} \label{lemma:bias} Let $\uu$ be the solution of the OCP \eqref{eqn:mother1} and assume $y(\uu), p(\uu) \in L^2_{\mathbb{P}}(\Gamma, H^{r+1}(D))$. Then, there exists $C(\uu)>0$ such that, $\forall h>0$:
\begin{equation*}
\E[\| \nabla f ^{h}(\uu,\cdot)-\nabla f (\uu, \cdot) \|^2] \leq C(\uu)h^{2r+2}.
\end{equation*}
\end{lemma}
\begin{proof}
A proof of this result can be found for example in \cite[Corollary 1]{MartinKrumscheidNobile}.
\end{proof}


\subsection{Multilevel Monte Carlo (MLMC) estimator}
\label{MLMC}
The OCP \eqref{eqn:mother1}, or its FE approximation \eqref{eqn:discr3}, involves computing the expectation of the cost functional $f(u, \omega)$. For this, in this work, we consider a Multilevel Monte Carlo estimator.
The key idea of MLMC \cite{Giles} is to estimate the mean of a random quantity $P := P(\omega)$ (in our case the last functional $f(u, \cdot)$), by simultaneously using Monte Carlo (MC) estimators on several approximations $P_l$, $l=0, \dots, L$ of $P$ of increasing accuracy (in our context these will be F.E. approximations on a hierarchy of computational meshes $\tau_{h_{\ell}}$ with mesh sizes $h_L < \dots < h_0$). Considering $N_l$ iid  realizations $\omega_{l,i}$, $i=1, \dots, N_l$ of the system's random input parameters on each level $0 \leq l \leq L$, the MLMC estimator for the expected value $\E[P]$ is given by
\begin{align} \label{Em}
\Emg [P]&=\frac{1}{N_{0}}\sum_{i=1}^{N_{0}}P_0(\omega_{0,i})+\sum_{l=1}^{L}\frac{1}{N_{l}}\sum_{i=1}^{N_{l}}\left( P_l(\omega_{l,i})-P_{l-1}(\omega_{l,i}) \right)\\
&=\sum_{l=0}^{L}\frac{1}{N_{l}}\sum_{i=1}^{N_{l}}\left( P_l(\omega_{l,i})-P_{l-1}(\omega_{l,i}) \right),
\end{align}
where we have set $P_{-1}=0$, and $\overrightarrow{N}=\{ N_{0}, \dots, N_{{L}} \}$.
We recall the main complexity result from \cite{Giles}. 
Denote by $V_l$ the variance of $P_l-P_{l-1}$ that is $V_l=\V[P_l-P_{l-1}]$ and by $C_l$ the expected cost of generating one realization of $(P_l,P_{l-1})$, $C_l=\E[Cost(P_l(\omega_{l,i}),P_{l-1}(\omega_{l,i}))]$.
If there exist positive constants $q_w, q_s, q_c, c_w, c_s, c_c$ such that $q_w \geq \frac{1}{2} \min(q_s, q_c)$ and
\begin{align}
&\big\vert\E[P_l-P]\big\vert \leq c_w 2^{- q_w l},\\
&V_l \leq c_s 2^{-q_s l},\\
& C_l \leq c_c 2^{q_c l},
\end{align}
\noindent then there exists a positive constant $c_0$ such that for any $\epsilon < e^{-1}$ there are values $L=L(\epsilon)$ and $N_l=N_l(\epsilon)$, $l=0, \dots, L(\epsilon)$, for which the multilevel estimator \eqref{Em} has a mean-squared-error MSE with bound
$$MSE := \E[\big\vert \Emeps[P] -\E[P] \big\vert^2] =  \big\vert \E[P_{L(\epsilon)} - P]\big\vert^2 + \sum_{\ell = 0}^L \frac{V_l}{N_l}$$
smaller or equal to $tol^2$ and an expected computational cost $C=\sum_{l=0}^L C_l N_l$ bounded by
$$C \leq \begin{cases}
      c_0 tol^{-2}, & \text{if}\ q_s>q_c, \\
      c_0 tol^{-2} \left(\log \epsilon \right)^2, & \text{if}\ q_s=q_c,\\
      c_0 tol^{-2-(q_c-q_s)/q_w}, & \text{if}\ q_s<q_c.
    \end{cases}$$
For given variances $V_l$, the optimal number of samples at level $l$ that minimize the cost $C_{l}$, is given by the formula:
$$N_{l}^{*}(\epsilon)=\left\lceil tol^{-2}\sqrt{\frac{V_l}{C_l}}\sum_{k=0}^{L(\epsilon)}\sqrt{V_kC_k}\right\rceil.$$


\section{Multilevel Stochastic Gradient (MLSG) algorithm}
\label{MLMC-SG}

First introduced by Robbins and Monro in 1961, Stochastic Approximation (SA) techniques, such as Stochastic Gradient (SG) \cite{robbins1951, PolJud, ShapiroStochasticProgramming, LeRouxBach, DieuBach} are popular iterative techniques in the machine learning community that are used to minimize an expected loss function $J(u) = \E[f(u)]$ over an admissible set $U$. The classic version of such a method, the so-called Robbins-Monro (RM) method, works as follows. Within the steepest descent algorithm, the exact gradient $\nabla J = \nabla \E[f]= \E[\nabla f]$\footnote{the fact that the gradient and the expectation can be exchanged for the OCP \eqref{eq:ocp} has been proven in \cite{MartinKrumscheidNobile}} is replaced by $\nabla f(\cdot, \omega_j )$, where the random variable $\omega_j$ (i.e. a random point $\omega_j \in \Gamma$ with distribution $\mathbb{P}$) is re-sampled independently at each iteration of the steepest-descent method:
\begin{equation} \label{SAiter}
u_{j+1} = u_j - \tau_j \nabla f(u_j , \omega_j ).
\end{equation}
Here, $\tau_j$ is the step-size of the algorithm and is decreasing as $\tau_0/(j + \alpha)$, where $\alpha$ denotes a shift parameter that is typically chosen in $\mathbb{N}$, with $\alpha = 0$ corresponding to the usual setting. The RM method applied to the OCP \eqref{eqn:mother1} has been analyzed in \cite{MartinKrumscheidNobile}. In this paper, we consider a generalization of this method, in which the "point-wise" gradient $\nabla f(\cdot, \omega_j )$ is replaced by a MLMC estimator, which is independent of the ones used in the previous iterations. Specifically, we introduce a hierarchy of refined meshes  $\tau_{h_0},\tau_{h_1}, \dots $ with    $ h_0>h_1>  \dots $ and corresponding FE spaces $Y^{h_l}$, $l=0, 1, \dots$, and $U^{h_l}=Y^{h_l}$ and use, at each iteration, the MLMC estimator defined in Section \ref{MLMC}, $\Em[P(u_j)]$, where $P(u)=\nabla f (u, \cdot)$ and $P_l(u)=\nabla f^{h_l}(u,\cdot)$. The total number of levels $L_j$ and samples per level $\overrightarrow{N}_j=\{ N_{j,0}, N_{j,1}, \dots, N_{{j,L_j}} \}$ are allowed to depend on $j$. In particular, we require $L_j$ and $N_{j,l}$ to be non decreasing sequences in $j$ and $L_j \rightarrow \infty$ as $j \rightarrow \infty$. For simplicity of exposition, we study only sequences of \textit{nested} meshes of size $h_l= h_0 2^{-l}$.
This implies, in particular, $Y^{h_{l-1}}\subset Y^{h_l}$, $\forall l \geq 1$. In this restricted context, the MLSG optimization procedure reads:
\begin{align*}\label{eq:MLSG_alg}
u_{j+1}&=u_j-\tau_j \Em [\nabla f(u_j,\cdot)],\\
&=u_j-\tau_j\left(\beta u_j+ \Em [p(u_j,\cdot)]\right) \quad \forall j \geq 1, \quad u_1 \in Y^{h_0},\numberthis
\end{align*}
with
\begin{equation} \label{MLMCestim}
\Em [p(u_j,\cdot)]=\sum_{l=0}^{L_j}\frac{1}{N_{j,l}} \sum_{i=1}^{N_{j,l}}\left( p ^{h_l}(u_j,\omega_{l,i}^{j})-p^{h_{l-1}}(u_j,\omega_{l,i}^{j}) \right).
\end{equation}
Notice that $u_j \in Y^{h_{L_{j-1}}}$ for any $j \geq 1$. This can be shown by a simple induction argument noticing that if $u_j \in Y^{h_{L_{j-1}}}$, then $\widehat{\nabla J}_{MLSG}:=\Em[\nabla f (u_j, \cdot)] \in Y^{h_{L_j}}$ which implies $u_{j+1} \in Y^{h_{L_j}}$ thanks to the nestedness assumption, $Y^{h_{p}}\subset Y^{h_{q}}$ if $p \leq q$.
\begin{remark}
In the case where non-nested meshes were used, we would need to introduce suitable interpolation operators $I_{h_p}^{h_q}:Y^{h_p} \rightarrow Y^{h_q}$ with $p<q$ and  modify the MLSG algorithm as follows
\begin{equation*}
u_{j+1}=(1-\beta \tau_j)I^{h_{L_{j}}}_{h_{L_{j-1}}}(u_j)-\tau_j\Em [p(u_j,\cdot)],
\end{equation*}
with
$$\Em [p(u_j,\cdot)]=\sum_{l=0}^{L_j} I_{h_l}^{h_{L_j}}\sum_{i=1}^{N_{j,l}}\frac{1}{N_{j,l}} \left( p ^{h_l}(u_j,\omega_{l,i}^{j})-I_{h_{l-1}}^{h_l}p^{h_{l-1}}(u_j,\omega_{l,i}^{j}) \right).$$
\end{remark}


\subsection{Convergence analysis}
\label{sec:sec31}

Let us denote by 
$$F_j=\sigma(\{ \omega_{l,i}^k\}, \hspace{1mm}k=1, \dots, j-1;\hspace{1mm} l=0, \dots, L_k;\hspace{1mm} i=1, \dots, N_{k,l} )$$ the $\sigma$-algebra generated by all the random variables $\{ \omega_{l,i}^k\}$ up to iteration $j-1$. Notice that the control $u_j$ is measurable with respect to $F_j$, i.e. the process $\{u_j\}$ is $\{F_j\}$-adapted.
We denote the conditional expectation to $F_j$ by $\E[\cdot|F_j]$.
\begin{theorem} 
\label{th10}
For any deterministic or random $F_j$-adapted sequence $\{ L_j,\hspace{1mm} j \geq 0\}$, $\{ N_{j,l} , \hspace{1mm} j \geq 0, \hspace{1mm} l \in \{ 0, \dots, L_j \}\}$, denoting by $u_{j}$ the approximated control obtained at iteration $j$ using the recursive formula (\ref{eq:MLSG_alg}), and $\uu$ the exact control for the continuous OCP \eqref{eqn:mother1}, we have:
\begin{equation}\label{mlmcSGRand}
\E[\| u_{j+1}-\uu \|^2|F_j]  \leq c_j\| u_j-\uu\|^2+ 2\tau_j^2 {\sigma_j^2}+\left( 2 \tau_j^2+\frac{2 \tau_j}{l} \right) {\epsilon_j^2}
\end{equation}
with
$c_j=1-\frac{l}{2}\tau_j+2Lip^2\left( 1+4\sum_{l=0}^{L_j} \frac{1}{N_{j,l}} \right)\tau_j^2$ \\
$\sigma_j^2=\sigma^2(L_j,\overrightarrow{N}_j )=\E[\|\Em [\nabla f (\uu)]-\E[\nabla f^{h_{L_j}} (\uu)|F_j]\|^2|F_j]\quad$ (variance term)\\ 
$\epsilon_j^2=\epsilon^2(L_j)=\|\E [\nabla f^{h_{L_j}} (\uu)|F_j]-\E[\nabla f  (\uu)]\|^2\quad$ (squared bias term).
\end{theorem}
\begin{proof}Using the optimality condition
$$\E[\nabla f  (\uu)]=0,$$ 
and the fact that the expectation of the MLMC estimator, for any deterministic or random $F_j$-measurable $u \in U$, is
$$\E[\Em [\nabla f (u)]|F_j]=\E[\nabla f^{h_{L_j}}(u)|F_j],$$
we can decompose the error at iteration $j+1$ as
\begin{equation}\label{first10tbis}
\begin{multlined}
\left\| u_{j+1}-\uu \right\|^2 =\Big\| u_j-\uu -\tau_j \underbrace{\left(\Em [\nabla f (u_j, \cdot)]-\E[\nabla f^{h_{L_j}} (u_j, \cdot)|F_j]\right)}_{=B} \\
 - \tau_j \underbrace{\left(\E[\nabla f^{h_{L_j}} (u_j, \cdot)|F_j]-\E[\nabla f^{h_{L_j}} (\uu, \cdot)|F_j]\right)}_{=A} -\tau_j \underbrace{\left(\E[\nabla f^{h_{L_j}} (\uu, \cdot)|F_j]-\E[\nabla f  (\uu, \cdot)]\right)}_{=C}\Big\|^2.
\end{multlined}
\end{equation}
Then, taking the conditional expectation of \eqref{first10tbis} to the $\sigma$-algebra $F_j$, we obtain the 10 following terms
\begin{equation}\label{eq:9teqbis}
\begin{multlined}  
\E[\| u_{j+1}-\uu \|^2|F_j]  =\| u_j-\uu\|^2 + \tau_j^2 \E[\| B \|^2|F_j] +\tau_j^2 \E[\| A \|^2|F_j]+\tau_j^2 \E[\| C \|^2|F_j]\\ \numberthis   -2 \tau_j \langle u_j-\uu,\E[B|F_j] \rangle -2 \tau_j \langle u_j-\uu,\E[A|F_j]\rangle-2 \tau_j \langle u_j-\uu,\E[C|F_j] \rangle \\  +2 \tau_j^2  \E[\langle A, C \rangle |F_j]+2 \tau_j^2 \E[\langle A,B \rangle|F_j]+2 \tau_j^2  \E[\langle B, C \rangle|F_j].
\end{multlined}
\end{equation}
The second term on the right hand side can be bounded using the expression of the variance of the MLMC estimator and the Lipschitz property in Lemma \ref{lemma:Lipconv_h}, as
\begin{align*}
\E[\| B \|^2|F_j] & =\E[\| \Em [\nabla f (u_j, \cdot)]-\E[\nabla f^{h_{L_j}} (u_j, \cdot)|F_j] \|^2|F_j]\\
& =\E\left[\left\| \sum_{l=0}^{L_j}\frac{1}{N_{j,l}} \sum_{i=1}^{N_{j,l}} \underbrace{\left(\nabla f^{h_l} (u_j,\omega_{i,l}^{j})-\nabla f^{h_{l-1}} (u_j,\omega_{i,l}^{j})-\E[\nabla f^{h_l} (u_j,\cdot)-\nabla f^{h_{l-1}} (u_j,\cdot)|F_j]\right)}_{=E_{i,l}(u_j)} \right\|^2|F_j\right]\\
& = \sum_{l=0}^{L_j} \sum_{i=1}^{N_{j,l}}\sum_{l'=0}^{L_j} \sum_{i'=1}^{N_{j,l'}}\frac{1}{N_{j,l}}\frac{1}{N_{j,l'}} \underbrace{\E\left[\langle{E_{i,l}}(u_j),{E_{i',l'}(u_j)}\rangle |F_j\right]}_{=0 \mathrm{\hspace{1mm}if\hspace{1mm}} i \neq i'\mathrm{\hspace{1mm}or\hspace{1mm}} l \neq l'}\\
& = \sum_{l=0}^{L_j} \sum_{i=1}^{N_{j,l}}\frac{1}{N_{j,l}^2} \E\left[\|E_{i,l}(u_j)\|^2 |F_j\right]\\
& = \sum_{l=0}^{L_j} \frac{1}{N_{j,l}} \E\left[\|E_{1,l}(u_j)\|^2 |F_j\right]\\
& \leq 2 \sum_{l=0}^{L_j} \frac{1}{N_{j,l}} \E[\| E_{1,l}(u_j)-E_{1,l}(\uu) \|^2|F_j]+2\sum_{l=0}^{L_j} \frac{1}{N_{j,l}}\E[\| E_{1,l}(\uu) \|^2|F_j] \\
& \leq 8 Lip^2 \left( \sum_{l=0}^{L_j} \frac{1}{N_{j,l}} \right) \| u_j-\uu \|^2+2\underbrace{\E[\| \Em[\nabla f (\uu, \cdot)]-\E[\nabla f^{h_{L_j}}(\uu, \cdot)|F_j] \|^2|F_j]}_{=\sigma_j^2}.
\end{align*}
In the above, for the first inequality we have used the identity $(a+b)^2 \leq 2a^2 + 2b^2$, for any $a, b \in \R$ while for the second we applied the Lipschitz bound on $f^{h_{l}}$.
Finally,
\begin{equation} \label{star1}
\tau_j^2\E[\| B \|^2|F_j] \leq \tau_j^2 8 Lip^2 \left( \sum_{l=0}^{L_j} \frac{1}{N_{j,l}} \right) \| u_j-\uu \|^2+2 \tau_j^2\sigma_j^2.
\end{equation}
Using again the Lischitz property, the third term in (\ref{eq:9teqbis}) can be bounded as
\begin{equation}\label{star2}
\tau_j^2\E[\| A \|^2|F_j] \leq \tau_j^2\E[\|\nabla f^{h_{L_j}} (u_j, \cdot)-\nabla f^{h_{L_j}} (\uu, \cdot)\|^2|F_j]
 \leq \tau_j^2 Lip^2 \| u_j-\uu \|^2,
\end{equation}
where we have exploited Jensen's inequality for conditional expectation to bound $\big\Vert\E[\cdot \vert F_j] \big\Vert^2 \leq \E[\Vert\cdot\Vert^2 \vert F_j]$.
The fourth term $\E[\| C \|^2|F_j]=\epsilon_j^2$ is just the squared bias term whereas the fifth term
$-2 \tau_j \langle u_j-\uu,\E[B|F_j] \rangle $ vanishes since $\E[B\vert F_j] = 0$.
For the sixth term, using strong convexity of $f^h$ uniform in $h$ (see Lemma \ref{lemma:Lipconv_h}), we get
\begin{align*}
-2 \tau_j \langle u_j-\uu,\E[A|F_j]\rangle & = -2 \tau_j \langle u_j-\uu,\E[\nabla f^{h_{L_j}} (u_j, \cdot)-\nabla f^{h_{L_j}} (\uu, \cdot)|F_j]\rangle\\
& \leq -l \tau_j \| u_j-\uu \|^2.
\end{align*}
The seventh term of (\ref{eq:9teqbis}) can be bounded as
\begin{align*}
-2 \tau_j \E\left[ \langle u_j-\uu,C \rangle|F_j \right] & \leq \frac{l \tau_j}{2} \| u_j-\uu \|^2+\frac{2 \tau_j}{l} \E\left[\| C \|^2|F_j \right]=\frac{l \tau_j}{2} \| u_j-\uu \|^2+\frac{2 \tau_j}{l} \epsilon_j^2.
\end{align*}
The eighth term can be decomposed as:
\begin{equation*}
2 \tau_j^2  \E[\langle A,C \rangle|F_j]  \leq \tau_j^2 \E[\| A \|^2|F_j]+\tau_j^2 \epsilon_j^2
\end{equation*}
The ninth and tenth terms vanish since $A$ and $C$ are measurable on $F_j$ and $\E[B|F_j]=0$. Hence
\begin{align*}
2 \tau_j^2 \E[\langle B,A \rangle|F_j]=2 \tau_j^2 \E[\langle B,C \rangle|F_j]=0.
\end{align*}
In conclusion, we have
\begin{equation}\label{bounds}
\E[\| u_{j+1}-\uu \|^2|F_j]  \leq \left( 1-\frac{l}{2}\tau_j+Lip^2\left( 2+8\sum_{l=0}^{L_j} \frac{1}{N_{j,l}} \right)\tau_j^2 \right)\| u_j-\uu\|^2+2 \tau_j^2 \sigma_j^2+\left( 2\tau_j^2+\frac{2}{l}\tau_j \right)\epsilon^2_j.
\end{equation}
\end{proof}
From now on, we consider only deterministic sequences $\{ L_j \}, \{  N_{j,l}\}$, i.e. chosen in advance and not adaptively during the algorithm. In this case the quantities $\sigma_j$ and $\epsilon_j$ defined in Theorem \ref{th10} are deterministic as well. From Theorem \ref{th10}, taking the full expectation $\E[\cdot]$ in \eqref{mlmcSGRand} leads to the recurrence
\begin{equation} \label{RecForm}
a_{j+1}  \leq c_j a_j+ \aaa \tau_j^2 \sigma_j^2+\bbb \tau_j \epsilon_j^2,
\end{equation}
where $a_j$ denotes the MSE $a_j=\E[\| u_j-\uu\|^2]$, $c_j$, $\sigma_j^2$, $\epsilon_j^2$ are defined in Theorem \ref{th10} and $\aaa=2$, $\bbb=2\tau_0+\frac{2}{l}$.

We now derive error bounds on the variance term $\sigma^2(L,\overrightarrow{N} )$ and the bias term $\epsilon^2(L)$ as a function of the total number of levels $L$ and  the sample sizes $\overrightarrow{N}=\{ N_{0}, N_{1}, \dots, N_{{L}} \}.$
\begin{lemma}\label{lemma:biaslemma} For a sufficiently smooth optimal control $\uu$ and primal and dual solutions $y(\uu)$, $p(\uu)$, the bias term $\epsilon^2(L)$ associated to the MLMC estimator \eqref{newrec} with $L$ levels satisfies 
$$\epsilon^2(L)  \leq C(\uu)h_L^{2r+2}$$
for some constant $C(\uu)>0$.
\end{lemma}
\begin{proof}
Following the definition of the MLMC estimator, and using its telescoping property, we have:
\begin{align*}
\epsilon^2(L) & =\| \E[\Emg[\nabla f (\uu)]]-\E[\nabla f (\uu)] \|^2\\
& = \|\E[\nabla f ^{h_L}(\uu,\cdot)]-\E[\nabla f (\uu, \cdot)] \|^2\\
& \leq \E[\| \nabla f ^{h_L}(\uu,\cdot)-\nabla f (\uu, \cdot) \|^2]\\
& \leq C(\uu)h_L^{2r+2}.
\end{align*}
where in the last step we have used Lemma \ref{lemma:bias}.
\end{proof}
\begin{lemma}\label{lemma:varlemma} For a sufficiently smooth optimal control $\uu$ and primal and dual solutions $y(\uu)$, $p(\uu)$, the variance term $\sigma^2(L,\overrightarrow{N} )$ associated to the MLMC estimator (\ref{eq:MLSG_alg}) using $L$ levels and $\overrightarrow{N}=\{ N_{0}, N_{1}, \dots, N_{{L}} \}$ sample sizes, satisfies 
$$\sigma^2(L,\overrightarrow{N} )  \leq 2 \widetilde{C}(\uu)\sum_{l=0}^L\frac{h_{l}^{2r+2}}{N_l}$$
for some constant $\widetilde{C}(\uu)>0$, namely $\widetilde{C}(\uu)=\max\{ C(\uu)\left( 1+2^{2r+2} \right), \frac{\E\left[ \left\|\nabla f^{h_0}(\uu)  \right\|^2 \right]}{2 h_0^{2r+2}} \}$ and $C(\uu)$ as in Lemma \ref{lemma:biaslemma}.
\end{lemma}
\begin{proof}
We use again the notation ${E_{i,l}(\uu)}=\nabla f^{h_l} (\uu,\omega_{i,l})-\nabla f^{h_{l-1}} (\uu,\omega_{i,l})-\E[\nabla f^{h_l} (\uu,\cdot)-\nabla f^{h_{l-1}} (u^*,\cdot)]$, with $\omega_{i,l}$ iid with distribution $\mathbb{P}$ on $\Gamma$ and denote $V_l=\E\left[\|E_{i,l}(\uu)\|^2\right]$. For $l=0, \dots, L$, $i=1, \dots, N_l$, we have
\begin{align*}
\sigma^2(L,\overrightarrow{N} )  =&\E \| \Emg[\nabla f (\uu, \cdot)]-\E[\nabla f^{h_L} (\uu, \cdot)] \|^2\\
 =& \E \| \sum_{l=0}^{L}\frac{1}{N_l} \sum_{i=1}^{N_l} E_{i,l}(\uu) \|^2\\
 =& \sum_{l,l'}\sum_{i,i'}\frac{1}{N_l N_{l'}}\E \langle E_{i,l}(\uu),E_{i',l'}(\uu) \rangle\\
 =& \sum_{l=0}^L\sum_{i=1}^{N_l}\frac{1}{N_l^2}\E[ \| E_{i,l}(\uu) \|^2]=\sum_{l=0}^L\frac{V_l}{N_l}
 \end{align*}
 with the following bounds for the quantity $V_0$
\begin{align*}
V_0 & \leq \E\left[ \left\| \nabla f^{h_0}(\uu, \cdot)-\E\left[ \nabla f^{h_0}(\uu, \cdot) \right] \right\|^2 \right]\\
&\leq \E\left[ \left\| \nabla f^{h_0}(\uu, \cdot) \right\| ^2\right],
\end{align*} 
and $V_l$, for $l \geq 1$:
\begin{align*}
V_l \leq& 2\E [\|  \nabla f^{h_l} (\uu,\cdot)-\nabla f (\uu,\cdot)-\E[\nabla f^{h_l} (\uu, \cdot)-\nabla f (\uu, \cdot)] \|^2] \\
&+2\E\left[\left\|\nabla f^{h_{l-1}} (\uu,\cdot)-\nabla f (\uu,\cdot)-\E\left[\nabla f^{h_{l-1}} (\uu, \cdot)-\nabla f (\uu, \cdot)\right] \right\|^2\right]\\
 \leq& 2\left(\E\left[ \|  \nabla f^{h_l} (\uu,\cdot)-\nabla f (\uu,\cdot) \|^2] +\E[\|\nabla f^{h_{l-1}} (\uu,\cdot)-\nabla f (\uu,\cdot) \|^2\right]\right)\\
 \leq& 2 C(\uu) \left( h_l^{2r+2}+h_{l-1}^{2r+2}\right)=2 C(\uu)\left( 1+2^{2r+2} \right)h_l^{2r+2}.
 \end{align*} 
Then combining the last two bounds, we reach:
\begin{align*}
\sigma^2(L,\overrightarrow{N} ) \leq& \sum_{l=0}^L\frac{2}{N_l}\widetilde{C}(\uu) h_{l}^{2r+2}
\end{align*}
with $\widetilde{C}(\uu)=\max\{ C(\uu)\left( 1+2^{2r+2} \right), \frac{\E\left[ \left\|\nabla f^{h_0}(\uu)  \right\|^2 \right]}{2 h_0^{2r+2}} \}$.
\end{proof}
From Lemmas \ref{lemma:biaslemma} and \ref{lemma:varlemma} we see that the bias $\epsilon_j^2=\epsilon^2(L_j)$ and the variance $\sigma^2_j=\sigma^2(L_j, \overrightarrow{N}_j)$ terms go to zero if $L_j, N_{j,l}\rightarrow \infty$ as $j \rightarrow \infty$. The rate at which $L_j$ and $N_{j,l}$ should diverge to $\infty$ as $j\rightarrow \infty$ should be chosen so as to minimize the total computational cost to achieve a prescribed accuracy.

We express here a final bound obtained on the MSE $a_j$ after $j$ iterations when using the MLSG iterative scheme in (\ref{eq:MLSG_alg}). Specifically, in the next Lemma we assume algebraic convergence rates for the two terms $\tau_j^2 \sigma_j^2\sim (j+\alpha)^{-\eta_1}$ and $\tau_j \epsilon_j^2\sim (j+\alpha)^{-\eta_2}$ and discuss in Lemma \ref{lemma:thmglobal} how the parameters $(\eta_1, \eta_2)$ should be chosen, to obtain the best complexity, while guaranteeing a given MSE.


\begin{lemma} \label{lemma:assumptionconvergence}
Assuming that we can choose $L_j$ and $\overrightarrow{N}_j$ such that the bias and variance terms decay as 
\begin{equation} \label{convepsilon}
C(\uu)h_{L_j}^{2r+2} \sim \epsilon_0^2 (j+\alpha)^{-\eta_2+1}, \quad \eta_2 \in ]1,\frac{\tau_0 l}{2}+1],
\end{equation}
\begin{equation}\label{convsigma}
\sum_{l=0}^{L_j}\frac{2}{N_{j,l}}\widetilde{C}(\uu) h_{l}^{2r+2}  \sim \sigma_0^2 (j+\alpha)^{-\eta_1+2}, \quad \eta_1 \in ]2,\frac{\tau_0 l}{2}+1],
\end{equation}
respectively, with $\sigma_0^2>0$, $\epsilon_0^2>0$ constants, and taking $\tau_j=\tau_0/(j+\alpha)$ with $\tau_0 l>2$, then we can bound the MSE $a_j=\E[\| u_j-\uu \|^2]$ for the MLSG scheme \eqref{eq:MLSG_alg} after $j\geq 2$ iterations as
\begin{equation} \label{rec2t}
  a_{j}\leq 
       C \begin{cases}
      (j+\alpha)^{1-\min\{\eta_1, \eta_2\}} \Big/(1 + \frac{\tau_0 l}{2} - \min\{\eta_1,\eta_2\}) , & \text{if}\ \min\{\eta_1,\eta_2\} < \frac{\tau_0 l}{2}+1 \\
      \log(j+\alpha)(j+\alpha)^{-\frac{\tau_0 l}{2}}, & \text{if}\ \eta_1 =\eta_2 = \frac{\tau_0 l}{2}+1
    \end{cases}
\end{equation}
with $C$ independent of $j$, $\eta_1$ and $\eta_2$.
\end{lemma}

\begin{proof}
We first notice that, as $N_{j,l}\geq 1$, under assumption \eqref{convepsilon},
\begin{equation}\label{sumNm1}
\sum_{l=0}^{L_j} \frac{1}{N_{j,l}} \leq L_j+1 \leq \widetilde{c} \log(j+1),
\end{equation}
for some constant $\widetilde{c}>0$.
From \eqref{RecForm} we obtain by induction
\begin{align*}
a_{j+1} \leq & c_j a_j + \aaa \tau_j^2 \sigma_j^2+\bbb \tau_j \epsilon_j^2 \\
\leq & c_j c_{j-1} a_{j-1} + c_j (\aaa \tau_{j-1}^2 \sigma_{j-1}^2+\bbb\tau_{j-1} \epsilon_{j-1}^2)+ \aaa\tau_j^2 \sigma_j^2+\bbb\tau_j \epsilon_j^2 \\
\lesssim & \cdots\\
\lesssim & \underbrace{\Big( \prod_{i=1}^j c_i \Big)}_{=\kappa_{0,j}} a_1 + \underbrace{\sum_{i=1}^j (\aaa\tau_i^2 \sigma_i^2+\bbb\tau_i \epsilon_i^2) \prod_{l=i+1}^j c_l}_{=\mathcal{B}_j}.\numberthis \label{eq:recbound}
\end{align*}
with $\aaa=2$ and $\bbb=2\tau_0+\frac{2}{l}$. For the first term $\kappa_{0,j}$, computing its logarithm, we have, 
\[
\log(\kappa_{0,j})=\sum_{i=1}^j \log(1-\frac{ \tau_0 l}{2(i+\alpha)}+ Lip^2\left(2+8\widetilde{c}\log(i+1)\right)\frac{\tau_0^2}{(i+\alpha)^{2}}) 
\leq \sum_{i=1}^j \frac{- \tau_0 l}{2(i+\alpha)} +\widehat{c}\sum_{i=1}^{j} \frac{\log(i+1)}{(i+\alpha)^{2}},
\]
with $\widehat{c}=Lip^2 \tau_0^2 \left( \frac{2}{\log 2}+ 8 \widetilde{c} \right)$. Thus
\[
\log(\kappa_{0,j}) \leq -\frac{\tau_0 l}{2} \log\left(j+1+\alpha\right) + \frac{\tau_0 l}{2}\log(\alpha+1) + M,\mathrm{\hspace{2mm}with\hspace{2mm}} M= \widehat{c}\sum_{i=1}^{\infty} \frac{\log(i+1)}{(i+\alpha)^{2}} < +\infty,
\]
where we have used the bound $\sum_{i=1}^j\frac{1}{i+\alpha} \leq \int_{1}^{j+1}\frac{1}{x+\alpha}dx$ and we finally get $\kappa_{0,j} \leq C_1 \left(j+\alpha\right)^{- \frac{\tau_0 l}{2}}$ and $C_1 = (\alpha+1)^{\tau_0 l /2}e^{M}$.
For the second term in \eqref{eq:recbound} we have:
\[
\mathcal{B}_j \leq \sum_{i=1}^j \left(\aaa \tau_0^2 \sigma_0^2 (i+\alpha)^{-\eta_1}+\bbb \tau_0 \epsilon_0^2 (i+\alpha)^{-\eta_2}\right) \underbrace{\prod_{k=i+1}^j c_k}_{=\kappa_{i,j}}.\]
For the term $\kappa_{i,j}$ we can proceed as above:
\begin{align*}
\log(\kappa_{i,j}) &\leq \sum_{k=i+1}^j \Big(-\frac{\tau_0 l}{2(k+\alpha)}+ \widehat{c}\frac{\log(k+1)}{(k+\alpha)^2}\Big) \\
&\leq - \frac{\tau_0 l}{2} \left(\log(j+1+\alpha)-\log(i+1+\alpha)\right)+M,
\end{align*}
which shows that
\begin{align*}
\kappa_{i,j} & \leq (j+1+\alpha)^{-\frac{\tau_0 l}{2}} (i+1+\alpha)^{\frac{\tau_0 l}{2}} e^M.
\end{align*}
It follows, in the case $\min\{\eta_1,\eta_2\}<\frac{\tau_0l}{2}+1$ and using $(i+1+\alpha)^{\tau_0l/2}\leq \left(2(i+\alpha)\right)^{\tau_0l/2}$ for $i \geq 1$, that 
\begin{align*}
\mathcal{B}_j & \leq  (j+1+\alpha)^{- \frac{\tau_0 l}{2}} e^M \sum_{i=1}^{j} \left(\aaa\tau_0^2 \sigma_0^2 (i+\alpha)^{\frac{\tau_0 l}{2}-\eta_1}+\bbb\tau_0 \epsilon_0^2 (i+\alpha)^{\frac{\tau_0 l}{2}-\eta_2}\right)2^{\tau_0 l/2}  \\
& \leq C_2 (j+\alpha)^{1-\min\{\eta_1,\eta_2\}} \big/(1 + \frac{\tau_0 l}{2} - \min\{\eta_1,\eta_2\}),
\end{align*}
with $C_2 = e^{M}\left(\lambda\tau_0^2\sigma_0^2 + \mu\tau_0\epsilon_0^2\right)2^{\tau_0l / 2}$,
whereas in the case $\eta_1=\eta_2=\frac{\tau_0l}{2}+1$ we have
\begin{align*} 
\mathcal{B}_j &\leq  \exp(M) (j+1+\alpha)^{- \frac{\tau_0 l}{2}} \sum_{i=1}^j  \left(\aaa\tau_0^2 \sigma_0^2 + \bbb\tau_0 \epsilon_0^2 \right)2^{\tau_0 l/2}(i+\alpha)^{-1}\\
&\leq  \exp(M)(j+1+\alpha)^{- \frac{\tau_0 l}{2}}\left(\aaa\tau_0^2 \sigma_0^2  + \bbb\tau_0 \epsilon_0^2 \right)\left(\log(j+1+\alpha) - \log(1+\alpha)\right)2^{\tau_0 l/2}\\
& \leq C_2 \log(j+\alpha) (j+\alpha)^{-\frac{\tau_0 l}{2}},
\end{align*}
for $j+\alpha \geq 2$, where we have exploited the fact that the function $\log x^{-\tau_0 l/2}$ is decreasing for $x\geq e^{2/\tau_0l}$.
Observing that the first term $\kappa_{0,j}a_1$ is always dominated by the second term $\mathcal{B}_j$, we obtain the final result with $C=\max\{C_1 a_1, C_2\}$.

\end{proof}
We introduce now some computational cost assumptions in order to assess the performance of the MLMC estimator introduced in Section \ref{MLMC}. Let us assume that the primal and dual problems can be solved using a triangulation with mesh size $h$, in computational time $C_h \lesssim h^{-d \gamma}$. Here, $\gamma \in [1,3]$ is a parameter representing the efficiency of the linear solver used (e.g. $\gamma=3$ for a direct solver and $\gamma=1$ up to a logarithmic factor for an optimal multigrid solver), while $d$ is the dimension of the physical space, $D \subset \R^d$. We can estimate the computational cost $C_l$ to generate one realization of $\nabla f^{h_l}(\uu, \cdot)-\nabla f^{h_{l-1}}(\uu, \cdot)$ by:
$$C_l=Cost(\nabla f^{h_l},\nabla f^{h_{l-1}}) \leq 2 Cost(\nabla f^{h_l})\lesssim h_l^{- \gamma d}\lesssim  2^{l \gamma d}.$$
On the other hand, in the particular context presented in this work and using the results in Lemma \ref{lemma:varlemma}, the variance $V_l$ at level $l$ can be bounded as
$$V_l=\E[\| \nabla f^{h_l}(\uu, \cdot)-\nabla f^{h_{l-1}}(\uu, \cdot) \|^2] \leq 2 \widetilde{C}(\uu)h_l^{2r+2}\lesssim2^{-l(2r+2)}. $$
Hence the costs and variances match the assumptions in Section \ref{MLMC} with $q_s=2r+2$ and $q_c=\gamma d$. For a problem in dimension $d \leq 3$ solved with an optimal solver ($\gamma=1$ up to log-terms) and using piecewise linear finite elements ($r=1$), we fall in the case $q_s>q_c$ and we should expect the MLMC estimator with optimal sample sizes to achieve an optimal complexity $tol^{-2}$ at each iteration. We show in Theorem \ref{corr:complexity2} that this optimality is preserved when MLMC is used within a stochastic gradient method to solve the OCP \eqref{eqn:mother1} when the rates $\eta_1, \eta_2$ are properly chosen. We start with a preliminary result.


\begin{lemma}\label{lemma:thmglobal}
For $\gamma$ such that the cost $C_l$ of computing one realization of $\nabla f^{h_l}(u, \cdot)-\nabla f^{h_{l-1}}(u, \cdot)$ is bounded by $C_l \lesssim 2^{l \gamma d}$, and choosing
\begin{itemize}
\item the step-size $\tau_j=\frac{\tau_0}{j+\alpha}$ with $\tau_0>\frac{2}{l}$; 
\item the sequence of levels $\{L_j\}_j$ such that
$$L_j=\left\lceil \frac{-1}{\log(2)} \log\left( \frac{1}{h_0}\left(\frac{\epsilon_0^2 (j+\alpha)^{1-\eta_2}}{ C(\uu)}\right)^{\frac{1}{2r+2}} \right)    \right\rceil,$$ for some $\eta_2 \in ] 1,  \frac{\tau_0l}{2}+1]$ and $\epsilon_0^2 < h_0^{2r+2}C(u^*)$, so that the bias term in \eqref{convepsilon} satisfies $\epsilon_j^2 \leq \epsilon_0^2 (j+\alpha)^{1-\eta_2}$, for all $j$;
\item the sequence of sample sizes $\{N_{j,l}\}_{\{j,l\}}$ as
$$N_{j,l}=\left\lceil \Upsilon(L_j) 2^{-l\frac{2r+2+\gamma d}{2}} (j+\alpha)^{\eta_1-2}\right\rceil$$
with $\Upsilon(L_j)=2 \widetilde{C}(\uu)\sigma_0^{-2} h_0^{2r+2}\sum_{k=0}^{L_j}2^{-k\frac{2r+2-\gamma d}{2}}$ for some $\eta_1 \in ] 2, \frac{\tau_0l}{2}+1 ]$ so that the variance term in \eqref{convsigma} satisfies $\sigma_j^2 \leq \sigma_0^2 (j+\alpha)^{2-\eta_1}$, for all $j$;
\end{itemize}
then, 
using the MLMC estimator $\Emi$ in (\ref{eq:MLSG_alg}) at each iteration $i=1, \dots, j$, the total required computational work $W_j$ to compute $u_j$, is bounded by:
\begin{eqnarray}
W_j \lesssim \frac{2^\phi}{\phi}(j+\alpha)^{\phi}\left(1 + (\eta_2-1)\log(j+\alpha)\right)^{\delta}
\end{eqnarray}
with 
\begin{eqnarray}
\phi = \left\{\begin{array}{ll} \max\{\eta_1-2, (\eta_2-1)\frac{\gamma d}{2r+2}\} + 1,\ \delta = 0, & \text{for}\ 2r+2 > \gamma d \\
\max\{\eta_1-1, \eta_2\},\ \delta = 2, & \text{for}\ 2r+2 = \gamma d \\ 
\max\{\eta_1-\eta_2-1 + \frac{(\eta_2-1)\gamma d}{2r+2}, \frac{(\eta_2-1)\gamma d}{2r+2}\} + 1,\ \delta = 0, & \text{for}\ 2r+2 < \gamma d \end{array} \right.,
\end{eqnarray}
where the hidden constant does not depend on 
$\eta_1$, $\eta_2$.
\end{lemma}
\begin{proof}
For the MLMC estimator $E_{L_j, \overrightarrow{N}_j}^{\textrm{MLMC}}$ with optimal $\overrightarrow{N}_j$ and $L_j$, the computational cost $\widehat{W}_j$ at iteration $j$ can be bounded as
\begin{align*} \label{costsingle}
\widehat{W_j} =\sum_{l=0}^{L_j}C_lN_{j,l} & \lesssim \sum_{l=0}^{L_j} 2^{l \gamma d}\left( 1+\left(2^{-l \frac{2r+2+\gamma d}{2}}(j+\alpha)^{\eta_1-2}\sum_{k=0}^{L_j} 2^{-k \frac{2r+2-\gamma d}{2}}\right)\right)\\
&\lesssim (j+\alpha)^{\eta_1-2}\left(\sum_{k=0}^{L_j}2^{-k\frac{2r+2-\gamma d}{2}}\right)^2+\sum_{l=0}^{L_j}2^{l \gamma d}\\
&\lesssim (j+\alpha)^{\eta_1-2}  \left(\sum_{k=0}^{L_j} 2^{-k \left(\frac{2r+2-\gamma d}{2}\right)}\right)^2+2^{\gamma d L_j},
\end{align*}
with hidden constant independent of $\eta_1, \eta_2$.
The first inequality just recalls that $N_{j,l}$ must be an integer, so the optimal value is rounded up to the nearest integer and can not be smaller than 1. Next, we consider the following cases separately: \\
a) $2r+2 > \gamma d$:
\begin{align*}
\widehat{W}_j & \lesssim (j+\alpha)^{\eta_1-2}+2^{\gamma d L_j}\\
& \lesssim (j+\alpha)^{\eta_1-2}+(j+\alpha)^{\frac{(\eta_2-1)\gamma d}{2r+2}}
\end{align*}
and the total cost for computing $u_j$ is
$$W_j=\sum_{i=1}^{j} \widehat{W_i} \lesssim \frac{2^{\phi}}{\phi}{(j+\alpha)}^{\phi},$$
with $\phi = \max \{ \eta_1-2,(\eta_2-1)\frac{\gamma d}{2r+2} \}+1$.\\
b) $2r+2 = \gamma d$:
\begin{align*}
\widehat{W}_j & \lesssim (j+\alpha)^{\eta_1-2} (L_j+1)^2 + 2^{\gamma d L_j} \\
& \lesssim (j+\alpha)^{\max\{\eta_1-1,\eta_2 \} - 1}\left(1 + \log (j+\alpha)^{\eta_2 - 1}\right)^2,
\end{align*}
which gives a total cost 
$$W_j =\sum_{i=1}^{j} \widehat{W_i} \lesssim \frac{2^\phi}{\phi}(j+\alpha)^{\phi}\left(1 + \log (j+\alpha)^{\eta_2-1}\right)^2,$$
with $\phi = \max\{\eta_1-1, \eta_2\}$.\\
c) $2r+2 < \gamma d$:
\begin{align*}
\widehat{W}_j & \lesssim (j+\alpha)^{\eta_1-2} 2^{L_j(\gamma d -2r - 2)} + 2^{\gamma d L_j} \\
& \lesssim (j+\alpha)^{\eta_1-2}(j+\alpha)^{(\eta_2-1)\frac{\gamma d -2r -2}{2r+2}} + (j+\alpha)^{\frac{(\eta_2-1)\gamma d}{2r+2}} \\
& \lesssim (j+\alpha)^{\phi - 1},
\end{align*}
with $\phi = \max\{\eta_1-\eta_2-1 + \frac{(\eta_2-1)\gamma d}{2r+2}, \frac{(\eta_2-1)\gamma d}{2r+2}\} + 1$ and the total cost becomes 
$$ W_j = \sum_{i=1}^j \widehat{W}_i \lesssim \frac{2^{\phi}}{\phi}(j+\alpha)^{\phi},$$
which completes the proof.
\end{proof}

We are now ready to analyze the asymptotic complexity of the MLSG scheme \eqref{eq:MLSG_alg} as well as the optimal choice of parameters $(\eta_1, \eta_2)$ within the set $\mathcal{S}=]2,\frac{\tau_0 l}{2}+1]\times ]1,\frac{\tau_0 l}{2}+1]$. For this we choose the interation count $j$ so that the bound on the MSE $a_j= \E[\| u_j-\uu \|^2]$ given in Lemma \ref{lemma:assumptionconvergence} is smaller than $tol^2$, i.e. setting $\psi(\eta_1,\eta_2)=\min\{\eta_1,\eta_2\}-1$ and $\mathcal{S}_{ad}:=\mathcal{S}\setminus \left\{\left(\frac{\tau_0 l}{2}+1,\frac{\tau_0 l}{2}+1\right)\right\}$, we take
        \begin{align*}
          j + \alpha &\sim \left(\frac{tol^{2} (\frac{\tau_0 l}{2}-\psi(\eta_1,\eta_2))}{C}\right)^{-\frac{1}{\psi(\eta_1,\eta_2)}}, && (\eta_1,\eta_2)\in \mathcal{S}_{ad},\\
          j+\alpha &\sim  \left(\frac{\tau_0 l}{8C} \frac{tol^{2}}{\log(tol^{-1})}\right)^{-\frac{2}{\tau_0 l}}, && \eta_1=\eta_2= \frac{\tau_0 l}{2}+1.
        \end{align*}
        
        With this choice of $j$ we have the following bound on the overall cost from Lemma \ref{lemma:thmglobal}:
        \begin{itemize}
          \item in the case $(\eta_1,\eta_2)\in \mathcal{S}_{ad}$
          \begin{equation}\label{eq:cost_sad}
          \overline{\mathcal{W}}(tol;\eta_1,\eta_2) = tol^{-\frac{2\phi(\eta_1,\eta_2)}{\psi(\eta_1,\eta_2)}} B(\eta_1,\eta_2)\left(C(\eta_1,\eta_2)+\frac{\eta_2-1}{\psi(\eta_1,\eta_2)}\log(tol^{-2})\right)^\delta,
          \end{equation}
          with $B(\eta_1,\eta_2)=\frac{2^\phi}{\phi}\left(\frac{\frac{\tau_0 l}{2}-\psi}{C}\right)^{-\frac{\phi}{\psi}}$ and $C(\eta_1,\eta_2)=1-\frac{\eta_2-1}{\psi(\eta_1,\eta_2)}\log(\frac{\frac{\tau_0 l}{2}-\psi}{C})$.
        \item in the case $\eta_1=\eta_2=\frac{\tau_0 l}{2}+1$
          \begin{equation} \label{eq:cost_corner}
            \overline{\mathcal{W}}(tol;\eta_1,\eta_2) = \left(tol^{-2} \log(tol^{-1})\right)^{\frac{\phi(\eta_1,\eta_2)}{\psi(\eta_1,\eta_2)}}\tilde{B}\left(\tilde{C} + \log(tol^{-2}\log(tol^{-1}))\right)^\delta,
            \end{equation}
            with $\tilde{B} = \frac{2^\phi}{\phi}\left(\frac{\tau_0 l}{8C}\right)^{-\frac{\phi}{\psi}}$ and $\tilde{C} = 1- \log(\frac{\tau_0 l}{8C})$.
          \end{itemize}
        We now say that a choice $(\eta_1^*,\eta_2^*)$ is \emph{asymptotically optimal} if
        \begin{equation}\label{eq:optimality}
        \liminf_{tol\to 0} \frac{\overline{\mathcal{W}}(tol;\eta_1,\eta_2)}{\overline{\mathcal{W}}(tol;\eta_1^*,\eta_2^*)} \ge 1, \qquad \forall (\eta_1,\eta_2)\in\mathcal{S}.
        \end{equation}
        The next theorem provides an asymptotically optimal choice of $(\eta_1,\eta_2)$ and states the resulting complexity of the MLSG scheme \eqref{eq:MLSG_alg}.
        



\begin{theorem}\label{corr:complexity2}
With the same notation and assumptions as in Lemma \ref{lemma:thmglobal}, given $\tau_0 > 2 / l$, the following choices of $(\eta_1, \eta_2)$ are asymptotically optimal when $tol \to 0$, according to definition \eqref{eq:optimality}, leading to the corresponding computational work $\mathcal{W}(tol)$ of the MLSG scheme (\ref{eq:MLSG_alg}), required to reach $MSE = O(tol^2)$:\\
a1) if $2r+2 > \gamma d$ and $\tau_0 > \frac{2}{l}\frac{2r+2}{2r+2-\gamma d}$, choose 
\begin{equation*}
\eta_1 = \eta_2 = 1 + \frac{2r+2}{2r+2 - \gamma d},
\end{equation*}
leading to $$\mathcal{W}(tol) \lesssim tol^{-2}.$$
a2) if $2r+2 > \gamma d$ and $\frac{2}{l} < \tau_0 < \frac{2}{l}\frac{2r+2}{2r+2-\gamma d}$, choose 
\begin{eqnarray*}
\begin{array}{ll} \eta_1 = \eta_2 = \frac{\tau_0 l}{2} + 1, \end{array}
\end{eqnarray*}
leading to 
$$\mathcal{W}(tol) \lesssim \left(tol^{-2}\log tol^{-1} \right)^{\frac{2}{\tau_0 l}+ \frac{\gamma d}{2r+2}}.$$
b) if $2r+2 = \gamma d$, choose 
\begin{eqnarray*}
\begin{array}{ll}\eta_1 = \eta_2 =  \frac{\tau_0l}{2} + 1\end{array},
\end{eqnarray*}
leading to $$\mathcal{W}(tol) \lesssim \left(tol^{-2}\log tol^{-1}\right)^{1 + \frac{2}{\tau_0l}} \left( \log \left( tol^{-2} \log (tol^{-1})\right)\right)^2.$$
c) if $2r+2 < \gamma d$, choose 
\begin{eqnarray*}
\begin{array}{ll}\eta_1 = \eta_2 =  \frac{\tau_0 l}{2} + 1\end{array},
\end{eqnarray*}
leading to $$\mathcal{W}(tol) \lesssim tol^{-2\left(\frac{2}{\tau_0l-2} + \frac{\gamma d}{2r+2}\right)}.$$
\end{theorem}
\begin{proof}
We want to find the best choice $(\eta_1, \eta_2) \in \Sad$ that minimizes the total computational work $\mathcal{W}(tol; \eta_1, \eta_2)$ for $(\eta_1, \eta_2) \in \Sad$. This translates to solving the optimization problem
\begin{equation}\label{eq:cost_a1}
\inf\limits_{ (\eta_1, \eta_2) \in \Sad} tol^{-\frac{2\phi(\eta_1,\eta_2)}{\psi(\eta_1,\eta_2)}} B(\eta_1,\eta_2)\left(C(\eta_1,\eta_2)+\frac{\eta_2-1}{\psi(\eta_1,\eta_2)}\log(tol^{-2})\right)^\delta.
\end{equation}
We simplify this optimization problem by minimizing only the leading exponent 
\begin{equation}
\inf\limits_{(\eta_1,\eta_2) \in \Sad} \frac{\phi(\eta_1, \eta_2)}{\psi(\eta_1, \eta_2)},
\end{equation}
This can be equivalently written as
\begin{equation} \label{eq:phioverpsi}
\inf\limits_{c \in \left]0,\frac{\tau_0 l}{2}\right[} \inf\limits_{\substack{(\eta_1, \eta_2) \in \Sad \\ \psi(\eta_1, \eta_2)=c}}\frac{\phi(\eta_1, \eta_2)}{c}.
\end{equation}
In cases \emph{a1)} and \emph{a2)} the problem reads
\begin{equation}
\label{eq:cost_a1a2}
\inf\limits_{c \in \left]0,\frac{\tau_0 l}{2}\right[} \inf\limits_{\substack{(\eta_1, \eta_2) \in \Sad \\ \min \{\eta_1, \eta_2 \}-1=c}}\frac{1+\max\{ \eta_1-2,(\eta_2-1)\frac{\gamma d}{2r+2} \}}{c}.
\end{equation}

\begin{figure}[!ht]
\centering
\includegraphics[width = 0.6\textwidth]{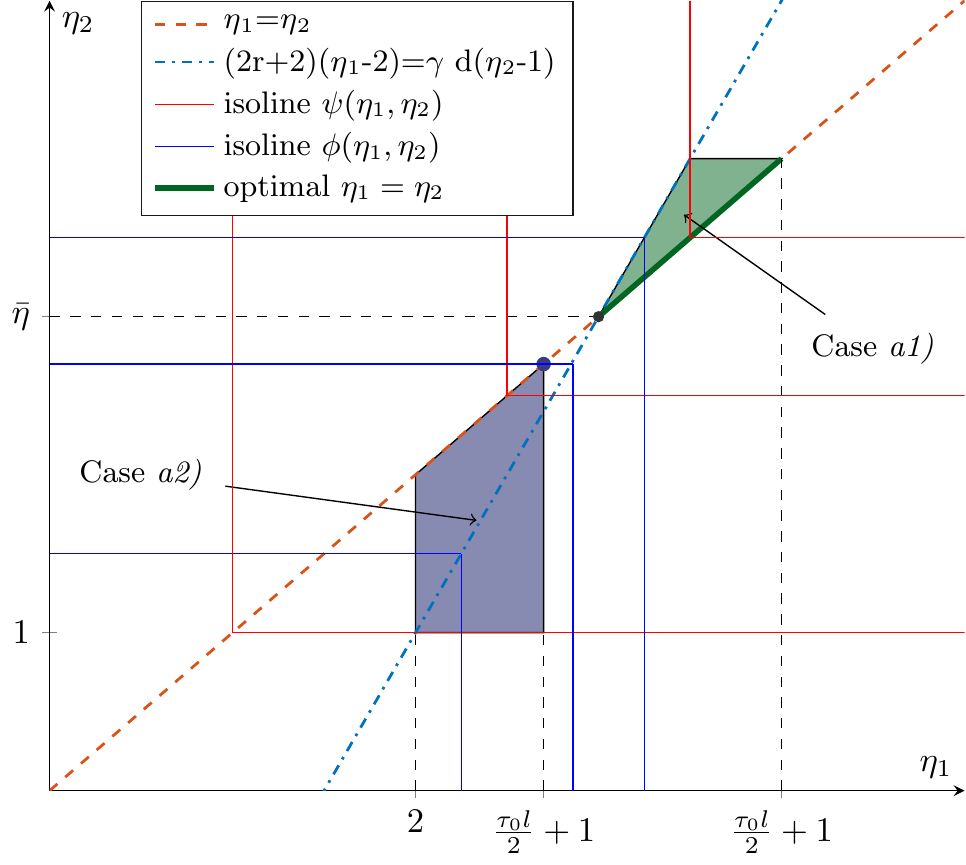}
\caption{Isocurves for problem \eqref{eq:phioverpsi}, cases \emph{a1)} and \emph{a2)} on the plane $(\eta_1, \eta_2)$.}
\label{fig:isocurves_i}
\end{figure}

Figure \ref{fig:isocurves_i} shows the isolines of the function $(\eta_1, \eta_2) \mapsto \psi(\eta_1,\eta_2)$ in solid red, which are L-shaped with corners on the bisectrice $\eta_1=\eta_2$ (red dashed line), as well as the isolines of the function $(\eta_1, \eta_2) \mapsto \phi(\eta_1,\eta_2)$ in solid blue, which have corners on the line
$$\eta_1-2=(\eta_2-1)\frac{\gamma d}{2r+2} \quad \Leftrightarrow \quad \eta_2=\underbrace{\frac{2r+2}{\gamma d}}_{>1}\eta_1+\underbrace{1-2\frac{2r+2}{\gamma d}}_{<-1}$$
(blue dashed line). The line crosses the bisectrice at $(\bar{\eta},\bar{\eta})$ (black ``$\bullet$") with $\bar{\eta} = 1 + \frac{\gamma d}{2r+2}$, and we distinguish the two different cases \emph{a1), a2)} based on whether the upper boundary of $\Sad$, that is $\frac{\tau_0 l}{2} + 1$ lies above or below $\bar{\eta}$ which yields 
$\tau_0 >\frac{2}{l}\frac{\gamma d}{2r+2}$ or $\tau_0 < \frac{2}{l}\frac{\gamma d}{2r+2}$ respectively. 
For the inner minimization, we see that on an isoline $\psi(\eta_1, \eta_2)=c$, the minimum of $\phi(\eta_1,\eta_2)$ is always achieved when $\eta_1=\eta_2$ (the minimizer is not unique and corresponds to a whole horizontal segment when $c<\bar{\eta}-1$, or a vertical segment when $c>\bar{\eta}-1$, see Figure \ref{fig:isocurves_i}). Finally \eqref{eq:cost_a1a2} can be rewritten as
\begin{equation} \label{prob40}
\inf\limits_{\eta \in \left]2,\frac{\tau_0 l}{2}+1\right[} \frac{1+\max\left\{ \eta-2, (\eta-1)\frac{\gamma d}{2r+2} \right\}}{\eta-1}=\inf\limits_{\eta \in \left]2,\frac{\tau_0 l}{2}+1\right[} \max\left\{ 1, \frac{1+(\eta-1)\frac{\gamma d}{2r+2}}{\eta-1} \right\}.
\end{equation}

Now, in case \emph{a1)} with $\frac{\tau_0 l}{2} + 1 > \bar{\eta}$ (i.e. $\tau_0 > \frac{2}{l}\frac{2r+2}{2r+2-\gamma d}$), then the optimum equals $1$ and is achieved by any $\eta \in \left[ \bar{\eta}, \frac{\tau_0l}{2}+1\right[$ (solid dark green). Using the fact that $\frac{\phi}{\psi} = 1$ and $\delta = 0$, we get $\mathcal{W}(tol) \lesssim \overline{\mathcal{W}}(tol;\eta_1, \eta_2) \lesssim tol^{-2}$ in the area $\Sigma_1 = \left\{(\eta_1, \eta_2): \bar{\eta} \leq \eta_1 < \frac{\tau_0l}{2}+1, \eta_1\leq\eta_2\leq\min\{1 + (\eta_1-2)\frac{2r+2}{\gamma d}, \frac{\tau_0l}{2}+1\}\right\}$ (green area in Fig. \ref{fig:isocurves_i}) and $\mathcal{W}(tol;\eta_1,\eta_2) \lesssim tol^{-\alpha}$, $\alpha > 2$ in $\Sad\setminus\Sigma_1$. Moreover, in $\Sigma_1$, the hidden constant $B(\eta_1, \eta_2) = \frac{2^{\eta_1-1}}{\eta_1-1}\frac{c}{\frac{\tau_0 l}{2}+1-\eta_1}$ is minimized for $\eta_1 = \eta_2 = \bar{\eta}$. Hence, $(\bar{\eta}, \bar{\eta})$ is the only asymptotically optimal choice, with complexity bound $\mathcal{W}(tol; \bar{\eta}, \bar{\eta}) \lesssim tol^{-2}$. From \eqref{eq:cost_corner} it is easy to check that $(\eta_1, \eta_2) = (\frac{\tau_0l}{2}+1, \frac{\tau_0l}{2}+1)$ is not an asymptotically optimal solution either.

On the other hand, in case \emph{a2)} with $2 < \frac{\tau_0l}{2} + 1 < \bar{\eta}$ that is, $\frac{2}{l} < \tau_0 < \frac{2}{l}\frac{2r+2}{2r+2-\gamma d}$, the infimum is achieved when $\eta \to \frac{\tau_0l}{2}+1$ (dark blue ``{\color{darkblue}$\bullet$}"). However, in the limit $\eta \to \frac{\tau_0l}{2}+1$, we have to use the expression for $\mathcal{W}(tol;\eta_1, \eta_2)$ given in \eqref{eq:cost_corner} with $\frac{\phi}{\psi} = \frac{2}{\tau_0l} + \frac{\gamma d}{2r+2}$ to obtain $\mathcal{W}(tol; \frac{\tau_0 l}{2}+1, \frac{\tau_0 l}{2}+1) \lesssim \left(tol^{-2} \log tol^{-1} \right)^{\frac{2}{\tau_0l} + \frac{\gamma d}{2r+2}}$, whereas $\mathcal{W}(tol, \eta_1, \eta_2) \lesssim tol^{-\alpha}$, $\alpha > 2$ for any $(\eta_1,\eta_2) \in \Sad$. Hence $\eta_1 = \eta_2 = \frac{\tau_0l}{2} + 1$ is the only asymptotically optimal solution. \\

For case \emph{b)} we have $\phi(\eta_1, \eta_2) = \max\{\eta_1-1, \eta_2\}$ and the optimization of the leading exponent reads
\begin{equation}
\inf\limits_{c \in \left]0,\frac{\tau_0 l}{2}\right[} \inf\limits_{\substack{(\eta_1, \eta_2) \in \Sad \\ \min \{\eta_1, \eta_2 \}-1=c}} \frac{\max\{\eta_1 - 1, \eta_2\}}{c}
\end{equation}
where the isolines of $\phi(\eta_1, \eta_2)$ have now corners on the line $\eta_2 = \eta_1 - 1$ while those of $\psi(\eta_1, \eta_2)$ are as before and the minimum is achieved on the horizontal line that joins the two lines. Taking again $\eta_1 = \eta_2 = \eta$ gives 
\begin{equation}\label{eq:prob_caseb}
\inf\limits_{\eta \in ]2, \frac{\tau_0l}{2}+1[ } \frac{\eta}{\eta-1} = 1 + \inf\limits_{\eta \in ]2, \frac{\tau_0l}{2}+1[ } \frac{1}{\eta-1}
\end{equation}
and the infimum is attained as $\eta \to \frac{\tau_0 l}{2} + 1$. Again, in the limit point, we use the MSE bound $a_j \lesssim \log(j + \alpha)(j + \alpha)^{-\frac{\tau_0 l}{2}}$ and the corresponding expression for the cost \eqref{eq:cost_corner} is 
\begin{equation}
\mathcal{W}(tol; \frac{\tau_0 l}{2} +1, \frac{\tau_0l}{2}+1) \lesssim \left(tol^{-2}\log tol^{-1}\right)^{1 + \frac{2}{\tau_0l}} \left(\tilde{C} + \log \left( tol^{-2} \log (tol^{-1})\right)\right)^2,
\end{equation}
whereas $\mathcal{W}(tol;\eta_1,\eta_2) \lesssim tol^{-\alpha}\left(\log tol^{-1}\right)^2$ with $\alpha > 2\left(1 + \frac{2}{\tau_0 l}\right)$ for $(\eta_1, \eta_2)\in \Sad$, hence $\eta_1 = \eta_2 = \frac{\tau_0 l}{2} + 1$ is the asymptotically optimal solution.

\begin{figure}[!h]
\centering
\includegraphics[width = 0.6\textwidth]{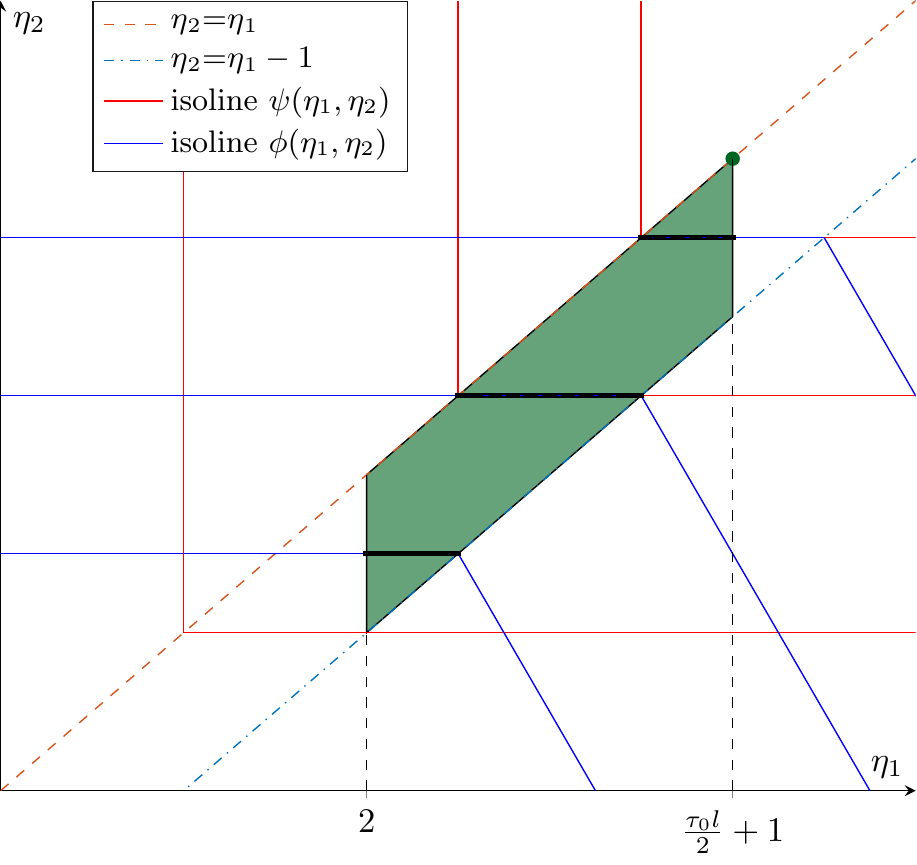}
\caption{Isocurves for problem \eqref{eq:phioverpsi}, \emph{c)} on the plane $(\eta_1, \eta_2)$.}
\label{isocurves_ii}
\end{figure}
At last, in case \emph{c)} the minimization problem becomes 
\begin{equation}
\inf\limits_{c \in \left]0,\frac{\tau_0 l}{2}\right[} \inf\limits_{\substack{(\eta_1, \eta_2) \in \Sad \\ \min \{\eta_1, \eta_2 \}-1=c}}\frac{1 + \max\{ \eta_1-\eta_2 - 1 + \frac{(\eta_2-1)\gamma d}{2r+2}, \frac{(\eta_2-1)\gamma d}{2r+2}\}}{c}.
\end{equation} 
Figure \ref{isocurves_ii} shows the isolines of the functions $(\eta_1, \eta_2) \mapsto \psi(\eta_1,\eta_2)$ and $(\eta_1, \eta_2) \mapsto \phi(\eta_1,\eta_2)$ as before, in red and blue respectively. The ones corresponding to $\phi(\eta_1, \eta_2)$ have corners on the line $ \eta_1 = \eta_2 + 1$ (blue dashed line). It is clear again, that for the $\psi(\eta_1, \eta_2) = c$, the minization of $\phi(\eta_1, \eta_2)$ is achieved on the horizontal segment that joins the lines $\eta_1 = \eta_2$ and $\eta_1 = \eta_2 + 1$ (black solid lines). By setting $\eta_1 = \eta_2$ we rewrite 
\begin{equation}\label{eq:prob45}
\inf\limits_{\eta_1 \in \left]2,\frac{\tau_0 l}{2}+1\right[} \frac{1 + \frac{(\eta_1 - 1)\gamma d}{2r+2} }{\eta_1 - 1} = \frac{\gamma d}{2r+2} + \inf\limits_{\eta_1 \in \left]2,\frac{\tau_0 l}{2}+1\right[}\frac{1}{\eta_1-1}
\end{equation}
which reaches the minimum value $\frac{\gamma d}{2r+2} + \frac{2}{\tau_0 l}$ at $\eta_1 \to \frac{\tau_0 l}{2}+1$. Using again the alternative estimate for the MSE bound, we obtain the cost expression
\begin{equation}
\mathcal{W}(tol;\frac{\tau_0 l}{2}+1, \frac{\tau_0 l}{2}+1) \lesssim \left(tol^{-2}\log tol^{-1}\right)^{1 + \frac{2}{\tau_0l}},
\end{equation}
whereas $\mathcal{W}(tol; \eta_1, \eta_2) \lesssim tol^{-\alpha}$, $\alpha > 2\left(1 + \frac{2}{\tau_0 l}\right)$ for any $(\eta_1, \eta_2) \in \Sad$, hence $\eta_1 = \eta_2 = \frac{\tau_0 l}{2} + 1$ is the asymptotically optimal solution.
 

\end{proof}
\begin{remark}
In the cases a2), b), c), one may conclude that the best complexity is achieved by taking $\tau_0 \to +\infty$. However, a closer inspection to the MSE bound given in Lemma \ref{lemma:assumptionconvergence} reveals that the hidden constant $C := C(\alpha, \tau_0)$ behaves as $C(\alpha, \tau_0) \lesssim \max\{\alpha+1, 2\}^{\tau_0l / 2}$ and grows exponentially in $\tau_0$. This leads to a trivial optimization of the parameters $(\alpha, \tau_0)$ which, however, would depend on the desired tolerance. This is not very convenient in practice and we do not pursue further this path.

\end{remark}

\subsection{Implementation of MLSG} \label{implementation}
Here we present an effective implementation of the MLSG algorithm, in the case $2r+2>\gamma d$.\\ \\ 
\begin{algorithm}[H]
\caption{MLSG algorithm}

\KwData{}
Choose $\eta \geq \frac{2(2r+2)-\gamma d}{2r+2-\gamma d}$, $\tau_0>\frac{2(\eta -1)}{l}, \alpha, h_0, \sigma_0, \epsilon_0$,\\
 generate the sequence $L_j=\left\lceil \frac{-1}{\log(2)} \log\left( \frac{1}{h_0}\left(\frac{\epsilon_0^2 j^{1-\eta}}{ C(\uu)}\right)^{\frac{1}{2r+2}} \right)    \right\rceil$,\\ 
 compute $N_{j,l}=\left\lceil\sigma_0^{-2} j^{\eta-2}2 \widetilde{C}(\uu)h_0^{2r+2}2^{-l\frac{2r+2+\gamma d}{2}}\sum_{k=0}^{L_j}2^{-k\frac{2r+2-\gamma d}{2}}\right\rceil.$\\ 
\textbf{Initialize } 
 $u=0$\;
 \For{$j \geq 1$}{ 
 \textbf{initialize } $\widehat{p}=0$\;
 generate $\prod_{l=0}^{L_j} N_{j,l}$ iid realizations of the random field $a_{l,i}^j=a(\cdot, \omega_{l,i}^j)$, $i=1,\dots, N_{j,l}$, $l=0,\dots, L_j$.\\
 \For{$l=0,\dots, L_j$}{
 \textbf{initialize } $\widehat{p_l}=0$\;
 \For{$i=1,\dots, N_{j,l}$}{
 solve primal problem by FE on mesh $h_{l-1}$ and realization\\ $a_{i,l}^j \rightarrow y^{h_{l-1}}(a_{i,l}^j,u)$ \\
 solve dual problem by FE on mesh $h_{l-1}$ and realization\\ $a_{i,l}^j\rightarrow p^{h_{l-1}}(a_{i,l}^j,y^{h_{l-1}})$\\
 solve primal problem by FE on mesh $h_l$ and realization\\ $a_{i,l}^j\rightarrow y^{h_l}(a_{i,l}^j,u)$ \\
 solve dual problem by FE on mesh $h_l$ and realization\\ $a_{i,l}^j\rightarrow p^{h_l}(a_{i,l}^j,y^{h_l})$\\
 update $\widehat{p_l} \leftarrow \widehat{p_l}+\frac{1}{N_{i,l}}\left( p^{h_l}(a_{i,l}^j,u)- p^{h_{l-1}}(a_{i,l}^j,u)\right)$\\
 }
 update $\widehat{p} \leftarrow \widehat{p}+\widehat{p_l}$\\
 }
 $\widehat{\nabla J} \leftarrow \beta u+ \widehat{p}$\\\
 $u \leftarrow u-\frac{\tau_0}{j+\alpha} \widehat{\nabla J}$\\
 }
\label{alg:algMLSG}
\end{algorithm}
The MLSG algorithm requires estimating the constants $C(\uu)$ and $\widetilde{C}(\uu)$ introduced in Lemmas \ref{lemma:biaslemma} and \ref{lemma:varlemma}. This could be done by a screening phase replacing the optimal control $\uu$ by e.g. the initial control $u_1$. Hence, for instance, for $M$ iid initial random inputs $\widetilde{\omega}_j, j=1, \dots, M$ distributed as $\mathbb{P}$ on $\Gamma$, we could estimate
$$\widehat{E}_l=\frac{1}{M}\sum_{j=1}^M\left\| \nabla f^{h_l}(u_1,\widetilde{\omega}_j)-\nabla f^{h_{l-1}}(u_1,\widetilde{\omega}_j) \right\|^2$$
and then approximate the constant $C(\uu)$ by least squares fit of the model $\widehat{E}_l=C(\uu) h_l^{2r+2}$. For the constant $\widetilde{C}(\uu)$ we have always taken in our simulation $\widetilde{C}(\uu)=C(\uu)$.
Notice that the choice of the constants $C(\uu)$ and $\widetilde{C}(\uu)$ does not affect the asymptotic complexity result.
A good choice of such constants, however, leads to a good balance of the error contributions in the MLMC estimator, notably its bias and variance.
On the same vein, the parameters $\sigma_0$ and $\epsilon_0$ should be chosen so that the two error contributions in the recurrence \eqref{RecForm}, namely $\aaa \tau_0^2 \sigma_0^2 j^{-\eta}$ and $\bbb \tau_0 \epsilon_0^2 j^{-\eta}$ are equilibrated. One such choice would be to fix $\sigma_0=\sqrt{\frac{\bbb}{\aaa  \tau_0}} \epsilon_0$.

\section{Randomized multilevel Stochastic Gradient (RMLSG) algorithm} \label{randomized}

We present in this section a modified version of the MLSG algorithm, that replaces the evaluation of the full MLMC estimator, at each iteration, by a single level estimator where the level used in the computation of the steepest direction is selected randomly. This corresponds to using at each iteration the randomized MLMC algorithm proposed in \cite{RheeGlynn2015,RheeGlynn2012} (see also \cite{Giles}). Specifically, at each iteration $j$, we sample an indice $l_j$ following a suitable discrete probability measure on $\{ 0, \dots, L_j \}$ with probability mass function $\{ p_l^j \}_{l=0}^{L_j}$ possibly changing at each iteration, and we set
\begin{equation} \label{randnewrec}
u_{j+1}=u_j-\tau_j \Er[\nabla f(u_j,\cdot)], \quad j \geq 1, \quad u_1 \in Y^{h_0}, 
\end{equation}
with
\begin{equation} \label{RandMLMC}
\Er[\nabla f(u_j,\cdot)]:=\frac{1}{p^j_{l_j}}\left(\nabla f^{h_{l_j}}(u_j,\omega_j)-\nabla f^{h_{l_j-1}}(u_j,\omega_j)\right),\quad l_j \sim  \{ p_l^j \}, \quad \omega_j \sim \mathbb{P},
\end{equation}
where all random variables $\left\{ l_j, \omega_j \right\}_{j \geq 1}$ are mutually independent.
We recall now from \cite{Giles, RheeGlynn2012, RheeGlynn2015} that the optimal choice for the discrete probability mass function $\{ p_l^j \}_{l=0}^{L_j}$, under the condition $q_s > q_c$ is
$
 p_l^{j}=2^{-l \frac{q_s+q_c}{2}}\left(\sum_{k=0}^{L_j}2^{-k \frac{q_s+q_c}{2}}\right)^{-1}
$
 which, in our setting with $q_s=2r+2$ and $q_c=\gamma d$ reads
\begin{equation} \label{probab}
 p_l^{j}=2^{-l \frac{2r+2+\gamma d}{2}}\left(\sum_{k=0}^{L_j}2^{-k \frac{2r+2+\gamma d}{2}}\right)^{-1}.
\end{equation}

\begin{remark}
The formula \eqref{probab} allows one to take $L_j=\infty$, $\forall j \geq 1$ when $2r+2 > \gamma d$. This leads to an unbiased estimator \eqref{RandMLMC}, and corresponds to the estimator proposed in \cite{RheeGlynn2012, RheeGlynn2015}. However, in this work, we prefer the biased version $L_j < \infty$, $\forall j \geq 1$ with $L_j$ suitably increasing in $j$, as it leads to a RMLSG algorithm with smaller variance of the computational cost (see Theorem \ref{thVar} below).
\end{remark}
The next Lemma quantifies the bias of the estimator \eqref{RandMLMC} for a fixed control $u$.
\begin{lemma} \label{randbiasbound}
For any $L_j \geq 0$ and any probability mass function $\left\{  p_l^j\right\}_{l=0}^{L_j}$ on $\{ 0, \dots, L_j \}$, we have
\begin{equation}
\E \left[ \Er[\nabla f(u,\cdot)] \right] = \E \left[ \nabla f^{h_{L_j}}(u,\cdot)\right].
\end{equation}
\end{lemma}
\begin{proof}
Conditioning on the value taken by the random variable $l_j$, we have
\begin{align*}
\E \left[ \Er[\nabla f(u,\cdot)] \right]&=\sum_{l=0}^{L_j} \E \left[\frac{1}{p^j_{l_j}}\left(\nabla f^{h_{l_j}}(u,\cdot)-\nabla f^{h_{l_j-1}}(u,\cdot)\right) |{{l_j}}=l \right] \underbrace{\mathbb{P}({{l_j}}=l)}_{=p_l^j}\\
&=\sum_{l=0}^{L_j} \E \left[ \nabla f^{h_l}(u,\cdot)-\nabla f^{h_{l-1}}(u,\cdot)  \right] \\
&=\E \left[ \nabla f^{h_{L_j}}(u,\cdot)\right].
\end{align*}
\end{proof}

So we conclude that the estimator \eqref{RandMLMC} has the same bias as the full MLMC estimator \eqref{MLMCestim}. 

\subsection{Convergence analysis}
Let us denote by 
$$F_j:=\sigma\left(\{ \omega_{i}\}, \{ l_i\}, \hspace{1mm}i=1, \dots, j-1
\right)$$ the $\sigma$-algebra generated by all the random variables $\{ \omega_{i}\}$ and the sampled indices $\{ l_i\}$ up to iteration $j-1$. 
Following the same procedure as in Section \ref{sec:sec31} we first derive a recurrence relation for the error $\| u_j- \uu \|$ at iteration $j$.
\begin{theorem} 
\label{th12}
For any deterministic or random $F_j$-adapted sequence $\{ L_j,\hspace{1mm} j \geq 1\}$, $\{ p_l^j , \hspace{1mm} j \geq 1, \hspace{1mm} l \in \{ 0, \dots, L_j \}\}$, with $\{ L_j \}$ non decreasing, denoting by $u_{j}$ the approximate control obtained at iteration $j$ using the recursive formula \eqref{randnewrec}, and $\uu$ the exact control for the continuous OCP \eqref{eqn:mother1}, we have:
\begin{equation}\label{mlmcSG}
\E[\| u_{j+1}-\uu \|^2|F_j]  \leq \overline{c_j}\| u_j-\uu\|^2+ 2\tau_j^2 {\overline{\sigma}_j^2}+\left( 2 \tau_j^2+\frac{2 \tau_j}{l} \right) {\overline{\epsilon}_j^2}
\end{equation}
with
$\overline{c_j}= 1-\frac{l}{2}\tau_j+2  Lip^2 \tau_j^2\left(4+3\sum_{l=0}^{L_j}\frac{2}{p^j_l} \right)$ \\
$\overline{\sigma}_j^2=\overline{\sigma}^2(L_j,\overrightarrow{p}^j )=\E\left[\left\|\Er [\nabla f (\uu)]-\E[\nabla f^{h_{L_j}} (\uu)|F_j]\right\|^2|F_j\right]$ the variance term\\ 
${\overline{\epsilon}_j^2}=\overline{\epsilon}^2(L_j)=\left\|\E [\nabla f^{h_{L_j}} (\uu)|F_j]-\E[\nabla f  (\uu)]\right\|^2$ the squared bias term.
\end{theorem}

\begin{proof}
Using the optimality condition
$$\E[\nabla f  (\uu)]=0,$$ 
and the fact that the expectation of the randomized MLMC estimator for any deterministic, or random $F_j$-measurable $u \in U$, is 
$$\E[\Er [\nabla f (u)]|F_j]=\E[\nabla f^{h_{L_j}}(u)|F_j],$$
we can decompose the error at iteration $j+1$ as
\begin{align*}
 \|u_{j+1}-\uu \|^2 & =  \left\|u_{j}-\uu-\tau_j \frac{1}{p^j_{l_j}}\left(\nabla f^{h_{l_j}}(u_j,\omega_j)-\nabla f^{h_{l_j-1}}(u_j,\omega_j)\right)\right\|^2 \\
&=  \left\|u_{j}-\uu-\tau_j \underbrace{\left( \frac{1}{p^j_{l_j}}\left(\nabla f^{h_{l_j}}(u_j,\omega_j)-\nabla f^{h_{l_j-1}}(u_j,\omega_j)\right) - \E \left[ \nabla f^{h_{L_j}}(u_j,\cdot)|F_j\right] \right)}_{=B(u_j)}\right.\\
& \left.-\tau_j \underbrace{\E \left[ \nabla f^{h_{L_j}}(u_j,\cdot)-\nabla f^{h_{L_j}}(\uu,\cdot)|F_j\right]}_{=A}-\tau_j \underbrace{\E \left[ \nabla f^{h_{L_j}}(\uu,\cdot)-\nabla f(\uu,\cdot)|F_j\right]}_{=C} \right\|^2\\
&= \| u_{j}-\uu \|^2 + \tau_j^2 \| A \|^2 + \tau_j^2 \| B(u_j) \|^2+ \tau_j^2 \| C \|^2\\
&-2\tau_j \langle u_{j}-\uu,A \rangle-2\tau_j \langle u_{j}-\uu,B(u_j) \rangle-2\tau_j \langle u_{j}-\uu,C \rangle\\
&+2\tau_j^2 \langle A,B(u_j) \rangle+2\tau_j^2 \langle A,C \rangle+2\tau_j^2 \langle B(u_j),C \rangle
\end{align*}
Taking conditional expectation w.r.t. the $\sigma$-algebra $F_j$ and using the fact that $\E[B(u_j)|F_j]=0$, which implies $\E[\langle u_{j}-\uu,B(u_j) \rangle|F_j]=\E[\langle A,B(u_j) \rangle|F_j]=\E[\langle C,B(u_j) \rangle|F_j]= 0$, we get
\begin{align*}
\E \left[ \|u_{j+1}- \uu \|^2  |F_j \right]  =&\| u_{j}-\uu \|^2 + \tau_j^2 \E[\| A \|^2|F_j] + \tau_j^2 \E[\| B(u_j) \|^2|F_j]+ \tau_j^2 \E[\| C \|^2|F_j]\\
&-2\tau_j \langle u_{j}-\uu,\E[A|F_j] \rangle-2\tau_j \langle u_{j}-\uu,\E[C|F_j] \rangle+2\tau_j^2 \E[\langle A,C \rangle|F_j]\\
 \leq & \| u_{j}-\uu \|^2 + \tau_j^2 \E[\| A \|^2|F_j] + \tau_j^2 \E[\| B(u_j) \|^2|F_j]+ \tau_j^2 \E[\| C \|^2|F_j]\\
&-\tau_j l \| u_{j}-\uu \|^2  +\frac{l}{2}\tau_j \| u_{j}-\uu \|^2 +\frac{2}{l}\tau_j \E[\| C \|^2|F_j]\\
&+\tau_j^2 \E[\| A\|^2|F_j]+\tau_j^2 \E[\|C \|^2|F_j]\\
 \leq & \left( 1-\frac{l}{2}\tau_j+2  Lip^2 \tau_j^2 \right) \| u_{j}-\uu \|^2 + \tau_j^2 \E[\| B(u_j) \|^2|F_j]+ \left(2\tau_j^2+\frac{2 \tau_j}{l}\right) \E[\| C \|^2|F_j],
\end{align*}
where we have exploited the Lipschitz continuity and strong convexity of $f^{h_{L_j}}$ to bound ${\E[\| A\|^2|F_j]}{\leq Lip^2 \| u_{j}-\uu \|^2}$ as well as $\langle u_{j}-\uu,\E[A|F_j] \rangle \geq \frac{l}{2}\| u_j-\uu \|^2$ (see also the proof of Theorem \ref{th10}). Splitting the term $B(u_j)$ as $ (B(u_j)- B(\uu))+B(\uu)$, we get:
\begin{align*}
\E \left[ \|u_{j+1}-\uu \|^2|F_j \right]  
\leq & \left( 1-\frac{l}{2}\tau_j+2  Lip^2 \tau_j^2 \right) \| u_{j}-\uu \|^2 + 2\tau_j^2 \E[\| B(u_j)-B(\uu) \|^2|F_j] \\
&+2\tau_j^2 \underbrace{\E[\| B(\uu) \|^2|F_j]}_{=\overline{\sigma}^{2}_j}+ \left(2\tau_j^2+\frac{2 \tau_j}{l}\right) \underbrace{\E[\| C \|^2|F_j]}_{=\overline{\epsilon}^{2}_j}.
\end{align*}
We can further split the term $B(u_j)-B(\uu)$ in 3 parts, and use the Lipschitz continuity:
\begin{align*}
\| B(u_j)- &B(\uu) \|^2  \leq  3\left\| \frac{1}{p^j_{l_j}}\left(\nabla f^{h_{l_j}}(u_j,\cdot)-\nabla f^{h_{l_j}}(\uu,\cdot)\right) \right\|^2\\
&+3\left\| \frac{1}{p^j_{l_j}}\left(\nabla f^{h_{l_j-1}}(u_j,\cdot)-\nabla f^{h_{l_j-1}}(\uu,\cdot)\right) \right\|^2+3\left\| \E[\nabla f^{h_{L_j}}(u_j,\cdot)-\nabla f^{h_{L_j}}(\uu,\cdot)|F_j] \right\|^2, 
\end{align*}
so that its conditional expectation reads
\begin{align*}
\E[\| B(u_j)- &B(\uu) \|^2|F_j]  \leq  3Lip^2 \left( 1+\sum_{l=0}^{L_j}\frac{2}{p^j_l} \right) \| u_{j}-\uu \|^2.
\end{align*}
Finally, we obtain:
\begin{equation}\label{bigrecRand}
\E \left[ \|u_{j+1}-\uu \|^2|F_j \right]  \leq \left( 1-\frac{l}{2}\tau_j+2  Lip^2 \tau_j^2\left(4+3\sum_{l=0}^{L_j}\frac{2}{p^j_l} \right) \right) \| u_{j}-\uu \|^2 + 2\tau_j^2 \overline{\sigma}^{2}_j+ \left(2\tau_j^2+\frac{2 \tau_j}{l}\right) \overline{\epsilon}^{2}_j.
\end{equation}
\end{proof}


From now on, we consider only deterministic sequences $\{ L_j \}, \{  p_l^j\}$, i.e. chosen in advance and not adaptively during the algorithm. In this case the quantities $\overline{\sigma}_j$ and $\overline{\epsilon}_j$ defined in Theorem \ref{th12} are deterministic as well. From Theorem \ref{th12}, taking the full expectation $\E[\cdot]$ in \eqref{bigrecRand} leads to the recurrence
\begin{equation} \label{recRand1}
a_{j+1}  \leq  \overline{c_j} a_j + \aaab \tau_j^2 \overline{\sigma}^{2}_j+ \bbbb \tau_j \overline{\epsilon}^{2}_j
\end{equation}
where $a_j$ denotes the MSE $a_j=\E[\| u_j-\uu\|^2]$, $\overline{c_j}$, $\overline{\sigma}_j^2$, $\overline{\epsilon}_j^2$ are defined in Theorem \ref{th12} and $\aaab=2$, $\bbbb=2\tau_0+\frac{2}{l}$ as in \eqref{RecForm}. Moreover, we restrict the analysis to the case $2r+2 > \gamma d$ so that $\sigma_j = o(1)$.
We now derive bounds on the bias term $\overline{\epsilon}^2(L)$ and the variance term $\overline{\sigma}^2(L,\{p_l\}_{l=0}^{L} )$  as a function of the total number of levels $L$ and  the 
probability mass function (pmf) $\{p_l\}_{l=0}^{L}$ on $\{ 0, \dots, L \}$. 
\begin{lemma} \label{biasbound}
For a sufficiently smooth optimal control $\uu$ and primal and dual solutions $y(\uu)$, $p(\uu)$, the bias term $\overline{\epsilon}^2(L)$ associated to the randomized MLMC estimator \eqref{RandMLMC} with $L$ levels satisfies 
$$\overline{\epsilon}^2(L)\leq C(\uu)h_{L}^{2r+2}.$$
with $C(\uu)$ as in Lemma \ref{lemma:bias}.
\end{lemma}
\begin{proof}
Since, from Lemma \ref{randbiasbound}, we have $\E \left[ \Erg[\nabla f(u,\cdot)] \right] = \E \left[ \nabla f^{h_{L}}(u,\cdot)\right]$, the proof follows verbatim that of Lemma \ref{lemma:biaslemma}.
\end{proof}

\begin{lemma} \label{biasvarbound}
For a sufficiently smooth optimal control $\uu$ and primal and dual solutions $y(\uu)$, $p(\uu)$, the variance term $\overline{\sigma}^{2}(L,\{ p_l \}_{l=0}^L )$ associated to the randomized MLMC estimator \eqref{RandMLMC}
using $L$ levels and 
 the pmf $\{ p_l \}_{l=0}^{L}$, satisfies
$$\overline{\sigma}^{2}(L,\{ p_l \}_{l=0}^L )\leq \sum_{l=0}^{L}2\widetilde{C}(\uu)\frac{h_{l}^{2r+2}}{p_l}$$
with $\widetilde{C}(\uu)$ as in Lemma \ref{lemma:varlemma}. If the pmf  $\{ p_l \}_{l=0}^{L}$ is chosen optimally as $p_l \propto 2^{-l\frac{2r+2+\gamma d}{2}}$, and $2r+2>\gamma d$ then 
$$\overline{\sigma}^{2}(L,\{ p_l \}_{l=0}^L )=O(1), \quad \mathrm{w.r.t. \hspace{1mm}}L.$$
\end{lemma}
\begin{proof}
For the first part of the Lemma, we start by showing that
\begin{equation*}
\overline{\sigma}^{2}(L,\{ p_l \}_{l=0}^L ) \leq \sum_{k=0}^{L} \frac{V_k}{p_k}
\end{equation*}
with $V_k=\E\left[\left\|\nabla f^{h_k}(\uu,\cdot)-\nabla f^{h_{k-1}}(\uu,\cdot) \right\|^2\right]$. We can write, following \cite{Giles}
\begin{align*}
\overline{\sigma}^{2}(L,\{ p_l \}_{l=0}^L ) = & \E \left[ \left\|\Erg[\nabla f(\uu,\cdot)]-\E \left[ \nabla f^{h_{L}}(\uu,\cdot)\right]\right\|^2 \right]\\
= & \E  \left[ \left\| \frac{1}{p_{l_R}}\left(\nabla f^{h_{l_R}}(\uu,\cdot)-\nabla f^{h_{l_R-1}}(\uu,\cdot)\right) \right\|^2 \right]-\left\| \E \left[ \nabla f^{h_{L}}(\uu,\cdot)\right] \right\|^2\\
\leq & \E  \left[ \left\| \frac{1}{p_{l_R}}\left(\nabla f^{h_{l_R}}(\uu,\cdot)-\nabla f^{h_{l_R-1}}(\uu,\cdot)\right) \right\|^2 \right]\\
\leq & \sum_{k=0}^{L} \frac{V_k}{p_k}.
\end{align*}
The variance terms $V_k$ can be bounded as in Lemma \ref{lemma:varlemma} as 
$$V_k \leq 2 \widetilde{C}(\uu) h_k^{2r+2}, \quad k=0, \dots, L$$
leading to the final result
$$\overline{\sigma}^{2}(L,\{ p_l \}_{l=0}^L ) \leq \sum_{k=0}^{L} \frac{2}{p_k} \widetilde{C}(\uu) h_k^{2r+2}.$$
Replacing the probability mass function $p_l \propto 2^{-l\frac{2r+2+\gamma d}{2}}$, as defined in \eqref{probab}, we get:
\begin{align*}
\overline{\sigma}^{2}(L,\{ p_l \}_{l=0}^L )\leq \sum_{l=0}^{L}2\widetilde{C}(\uu)\frac{h_{l}^{2r+2}}{p_l}&=\sum_{l=0}^{L}2\widetilde{C}(\uu)h_0^{2r+2}2^{-l(2r+2)}2^{l(\frac{2r+2+\gamma d}{2})}\frac{2^{-(L+1)\frac{2r+2+\gamma d}{2}}-1}{2^{-\frac{2r+2+\gamma d}{2}}-1}\\
& \leq 2\widetilde{C}(\uu)h_0^{2r+2} \sum_{l=0}^{L}2^{-l\frac{2r+2-\gamma d}{2}}\left({1-2^{-\frac{2r+2+\gamma d}{2}}}\right)^{-1}\\
&=O(1).
\end{align*}
Since $2r+2- \gamma d>0$ and the series is convergent. 
\end{proof}
Analogously to the non-randomized MLSG algorithm, we enforce an algebraic decrease of the bias as a function of $j$, i.e. 
\begin{equation}
C(\uu)h_0^{2r+2} 2^{-{L_j}{(2r+2)}} \sim \overline{\epsilon}^{2}_0 j^{1-\eta}, \quad \mathrm{for \hspace{1mm} some\hspace{1mm}}\eta>1.
\end{equation}
Under the further condition $\eta<\frac{3(2r+2)+\gamma d}{2r+2+\gamma d}$, we can obtain a bound on the MSE $a_j=\E[\| u_j-\uu \|^2]$ for the RMLSG algorithm, analogous to the one stated in Lemma \ref{lemma:assumptionconvergence} for the non-randomized version.
\begin{lemma} \label{lemma:assumptionconvergenceRand}
Assuming $2r+2>\gamma d$ and choosing $L_j$ such that the bias term decays as 
\begin{equation} \label{convepsilonRand}
C(\uu)h_{L_j}^{2r+2} \sim \overline{\epsilon}_0^2 j^{-\eta+1}, \quad \eta \in \left] 1, \frac{3(2r+2)+\gamma d}{2r+2+\gamma d} \right[,
\end{equation}
with $\overline{\epsilon}_0^2>0$ constant, taking $\tau_j=\tau_0/j$ with $\tau_0 l>2$, and the probability mass function  $\{p_l^{j}\}$ as in \eqref{probab}, then we can bound the MSE $a_j=\E[\| u_j-\uu \|^2]$ for the RMLSG algorithm \eqref{randnewrec} after $j$ iterations as
\begin{equation} \label{rec2tRand}
 a_{j}\leq C_1(a_1) j^{-\frac{\tau_0 l}{2}}+C_2 j^{1-\min\{ 2,\eta \}}
\end{equation}
with $C_1(a_1)$ and $C_2$ independent of $j$.
\end{lemma}
\begin{proof}
We notice that we can bound $\sum_{l=0}^{L_j}\left(p_l^{j}\right)^{-1}$ appearing in the constant $\overline{c_j}$ of Theorem \ref{th12}, by
\begin{equation}\label{sumpm1}
\sum_{l=0}^{L_j}\left(p_l^{j}\right)^{-1} \leq \widetilde{c} j^{\theta},
\end{equation}
for some constant $\widetilde{c}>0$ and $\theta={\frac{\eta-1}{2}{\left(1+\frac{\gamma d}{2r+2}\right)}}$. Indeed, denoting by $P_j=\sum_{l=0}^{L_j}2^{-l \frac{2r+2+\gamma d}{2}}=O(1)$, we have
\begin{align*}
\sum_{l=0}^{L_j}\left(p_l^{j}\right)^{-1}&=\sum_{l=0}^{L_j}2^{l \frac{2r+2+\gamma d}{2}}P_j\\
&\lesssim 2^{(L_j+1) \frac{2r+2+\gamma d}{2}}\\
&\lesssim j^{\frac{\eta-1}{2r+2}\frac{2r+2+\gamma d}{2}}\\
&\lesssim j^{\theta}.
\end{align*}
The condition $\eta<\frac{3(2r+2)+\gamma d}{2r+2+\gamma d}$ is equivalent to $\theta<1,$ which makes the series $\sum_{l=0}^{L_j} \frac{j^{-2}}{p_l^j}$ summable. From \eqref{recRand1} we obtain by induction
\begin{align*}
a_{j+1} \leq & \overline{c_j} a_j + \aaab \tau_j^2 \overline{\sigma}_j^2+\bbbb \tau_j \overline{\epsilon}_j^2 \\
\leq & \overline{c_j} \hspace{1mm}\overline{c_{j-1}} a_{j-1} + \overline{c_j} (\aaab \tau_{j-1}^2 \overline{\sigma}_{j-1}^2+\bbbb\tau_{j-1} \overline{\epsilon}_{j-1}^2)+ \aaab\tau_j^2 \overline{\sigma}_j^2+\bbbb\tau_j \overline{\epsilon}_j^2 \\
\lesssim & \cdots\\
\lesssim & \underbrace{\Big( \prod_{i=1}^j \overline{c_i} \Big)}_{=\kappa_{0,j}} a_1 + \underbrace{\sum_{i=1}^j (\aaab\tau_i^2 \overline{\sigma}_i^2+\bbbb\tau_i \overline{\epsilon}_i^2) \prod_{l=i+1}^j \overline{c_l}}_{=\mathcal{B}_j}.\numberthis \label{recboundRand}
\end{align*}
with $\aaab=2$ and $\bbbb=2\tau_0+\frac{2}{l}$. For the first term $\kappa_{0,j}$, computing its logarithm, we have,

\begin{equation*}
\log(\kappa_{0,j}) \leq\sum_{i=1}^j \log\left(1-\frac{ \tau_0 l}{2i}+ \tau_0^2 Lip^2\left(8+12\widetilde{c}i^{\theta}\right)i^{-2}\right)\leq \sum_{i=1}^j \frac{- \tau_0 l}{2i} +\widehat{c}\sum_{i=1}^{j} i^{\theta-2},
\end{equation*}
with $\widehat{c}=Lip^2 \left(8+12\widetilde{c}\right)\tau_0^2$. Hence,
\begin{equation*}
\log(\kappa_{0,j}) \leq -\frac{\tau_0 l}{2} \log(j+1) + M',\mathrm{\hspace{2mm}with\hspace{2mm}} M'=\widehat{c}\sum_{i=1}^{\infty} i^{\theta-2}< + \infty,
\end{equation*}
which implies $\kappa_{0,j} \lesssim j^{- \frac{\tau_0 l}{2}}$.
For the second term $\mathcal{B}_j$ in \eqref{recboundRand} we have:
\[
\mathcal{B}_j \lesssim \sum_{i=1}^j \left(\aaab \tau_0^2 i^{-2}+\bbbb \tau_0 \overline{\epsilon}_0^2 i^{-\eta}\right) \underbrace{\prod_{k=i+1}^j \overline{c_k}}_{=\kappa_{i,j}}.\]
For the term $\kappa_{i,j}$ we can proceed as follows:
\begin{align*}
\log(\kappa_{i,j}) & =\sum_{k=i+1}^j \log\left(\overline{c_k}\right) \\
& \leq \sum_{k=i+1}^j \log\left(1-\frac{\tau_0 l}{2k}+ \widehat{c}\frac{k^{\theta}}{k^2}\right) \\
&\leq \sum_{k=i+1}^j \Big(-\frac{\tau_0 l}{2k}+ \widehat{c}\frac{k^{\theta}}{k^2}\Big) \\
&\leq - \frac{\tau_0 l}{2} \left(\log(j+1)-\log(i+1)\right)+M',
\end{align*}
which implies
\begin{align*}
\kappa_{i,j} & \leq (j+1)^{-\frac{\tau_0 l}{2}} (i+1)^{\frac{\tau_0 l}{2}} \exp\left(M'\right),
\end{align*}
and the final bound on $B_j$:
\begin{align*}
\mathcal{B}_j & \lesssim  (j+1)^{- \frac{\tau_0 l}{2}} {\exp\left(M'\right)} \sum_{i=1}^{j} \left(\aaa\tau_0^2 i^{\frac{\tau_0 l}{2}-2}+\bbb\tau_0 \overline{\epsilon}_0^2 i^{\frac{\tau_0 l}{2}-\eta}\right)2^{\tau_0 l/2}  \\
& \lesssim j^{1-\min\{ 2,\eta \}}.
\end{align*}
\end{proof}
We are now ready to state the final complexity result for the RMLSG algorithm.
\begin{theorem}\label{corr:complexity2Rand}
In the case $2r+2> \gamma d$ with the same notation and assumptions as in Lemma \ref{lemma:assumptionconvergenceRand}, if the parameter $\eta$ satisfies 
$\eta \in \left] 1,  \frac{3(2r+2)+\gamma d}{2r+2+\gamma d}  \right[$, and 
$\tau_0>\frac{2 }{l}$, then the expected computational work $W(tol)$ of the RMLSG algorithm \eqref{randnewrec} to reach a MSE $O(tol^2)$, is bounded by:
\begin{equation}
\E[W(tol)]  \lesssim  tol^{\min\left\{ -2, \frac{-2}{\eta-1} \right\}}.
\end{equation}
In particular, if we choose $\eta \in \left[ 2,  \frac{3(2r+2)+\gamma d}{2r+2+\gamma d}  \right[$, we reach the optimal complexity  of $\E[W(tol)] \lesssim tol^{-2}$.
\end{theorem}
\begin{proof}
The expected computational work of RMLSG up to iteration $j$, namely $W_j$, can be bounded by
\begin{align*}
\E[W_j]=\sum_{i=1}^{j}\E[C_{l_i}]&=\sum_{i=1}^{j}\sum_{k=0}^{L_i}C_k p_k^{i}\\
& \lesssim \sum_{i=1}^{j}\sum_{k=0}^{L_i} 2^{k \gamma d} 2^{-k \frac{2r+2+\gamma d}{2}}\left(\sum_{p=0}^{L_j}2^{-p \frac{2r+2+\gamma d}{2}}\right)^{-1}\\
& = \sum_{i=1}^{j}\sum_{k=0}^{L_i} 2^{-k \frac{2r+2-\gamma d}{2}}\left(\sum_{p=0}^{L_j}2^{-p \frac{2r+2+\gamma d}{2}}\right)^{-1}\\
& \lesssim O(j).\\
\end{align*}
Taking now $j \sim tol^{\frac{2}{1-\min\{ 2, \eta \}}}$ to achieve $a_j \lesssim tol^2$ we finally obtain $\E[W(tol)] \lesssim tol^{\frac{2}{1-\min\{ 2, \eta \}}}$ which can be equivalently rewritten as
$$\E[W(tol)] \lesssim tol^{\min\left\{ -2, \frac{-2}{\eta-1} \right\}}.$$
\end{proof}
The previous Theorem shows that the RMLSG algorithm achieves the optimal complexity $\E[W(tol)] \lesssim tol^{-2}$ (in terms of expected computational cost versus MSE), for all $\eta \in \left[ 2,  \frac{3(2r+2)+\gamma d}{2r+2+\gamma d}  \right[$. It is worth looking also at the variance of the computational work, beside its expected value. The next Lemma shows that the choice $\eta=2$ is optimal in the sense that it minimizes the variance of the cost among all $\eta \in \left[ 2,  \frac{3(2r+2)+\gamma d}{2r+2+\gamma d}  \right[$, at least in the case $2r+2<3 \gamma d$, and leads to a coefficient of variation $\varrho(tol)=\frac{\sqrt{\V[W(tol)]}}{\E[W(tol)]}$ that goes asymptotically to zero.

\begin{theorem} \label{thVar}
Let $W(tol)$ be the computational cost to reach a MSE $=O(tol^2)$ by the RMLSG algorithm \eqref{randnewrec}, and denote by $\varrho(tol)$ the coefficient of variation of $W(tol)$, namely 
$$\varrho(tol)=\frac{\sqrt{\V[W(tol)]}}{\E[W(tol)]}.$$
Assuming that the computational cost $C_l$ of computing one realization of $\nabla f^{h_l}(u, \cdot)-\nabla f^{h_{l-1}}(u, \cdot)$ can be bounded as $\underline{c_c}2^{l \gamma d} \leq C_l \leq \overline{c_c}2^{l \gamma d}$ for some $\underline{c_c}, \overline{c_c}>0$,
then if $2r+2 \geq 3 \gamma d$,
$$\lim_{tol \rightarrow 0}\varrho(tol)=0, \quad \forall \eta \in \left[ 2, \frac{3(2r+2)+\gamma d}{2r+2+\gamma d} \right[.$$
On the other hand, if $\gamma d < 2r+2 < 3 \gamma d$,
$$\lim_{tol \rightarrow 0}\varrho(tol)=0, \quad \forall \eta \in \left[ 2, \frac{3 \gamma d+(2r+2)}{3 \gamma d-(2r+2)} \right[.$$
Moreover, $\varrho(tol)$ is minimized for $\eta=2$, for which
$$\varrho(tol) \lesssim tol^{\frac{3(2r+2-\gamma d)}{2(2r+2)}}.$$
\end{theorem}
\begin{proof}
Let us start by computing the variance of the computational cost $W_j$, after $j$ iterations
\begin{align*}
\V\left[W_j\right]:=\V\left[ \sum_{i=1}^{j}C_{l_i} \right]&= \sum_{i=1}^{j}\V\left[C_{l_i} \right]\\
&\leq \sum_{i=1}^{j} \E\left[ C_{l_i}^2 \right]\\
&=\sum_{i=1}^{j}\sum_{l=0}^{L_i}  C_{l}^2 p_l^i\\
&\lesssim \sum_{i=1}^{j}\sum_{l=0}^{L_i}  2^{2l\gamma d}2^{-l\frac{2r+2+\gamma d}{2}}\\
& \lesssim \sum_{i=1}^{j}\sum_{l=0}^{L_i}  2^{-l\frac{2r+2-3\gamma d}{2}}.
\end{align*}
Observe moreover, that under the assumption $C_l \sim 2^{l \gamma d}$, we have
\begin{align*}
\E\left[W_j\right]&=\sum_{i=1}^{j}\sum_{l=0}^{L_i}  C_{l} p_l^i\\
&\geq \sum_{i=1}^{j} \frac{1}{P_j}\sum_{l=0}^{L_i}\underline{c_c}2^{2l\gamma d}2^{-l\frac{2r+2+\gamma d}{2}}\\
& \geq \frac{\underline{c_c}}{P_{\infty}}j.
\end{align*}
If $2r+2 > 3 \gamma d$, the series $\{ \sum_{l=0}^{L_i}  2^{-l\frac{2r+2-3\gamma d}{2}} \}_i$ is convergent, we end up with
$$\V[W_j] \lesssim j, \quad \forall \eta \in \left[2, \frac{3(2r+2)+\gamma d}{2r+2+\gamma d}\right[.$$
In this case, the squared coefficient of variation $\frac{\V[W_j]}{\E[W_j]^2} \lesssim j^{-1} $ which implies $\varrho^2(tol) \lesssim tol^2$, $\forall \eta \in \left[2, \frac{3(2r+2)+\gamma d}{2r+2+\gamma d}\right[.$\\
If, instead, $2r+2 = 3 \gamma d$, then we have
$$\V[W_j] \lesssim \sum_{i=1}^{j}L_i \leq \sum_{i=1}^{j} \log \left(  i^{\frac{\eta-1}{2r+2}}\right) \leq \frac{\eta-1}{2r+2} (j+1) \log (j+1)$$
and then
  $$\varrho(tol)^2= \frac{\V[W(tol)]}{\E[W(tol)]^2} \lesssim \frac{\eta-1}{2r+2} tol^2 \log(tol^{-1})$$ 
 which is minimized for $\eta=2$.\\
 Finally, when $2r+2 < 3 \gamma d$, we have
 $$\V[W_j] \lesssim \sum_{i=1}^{j}2^{L_i\frac{3 \gamma d-(2r+2)}{2}}\lesssim \sum_{i=1}^{j} i^{\frac{\eta-1}{2r+2}\frac{3 \gamma d-(2r+2)}{2}} \lesssim j^{\frac{\eta-1}{2r+2}\frac{3 \gamma d-(2r+2)}{2}+1}$$
 and we derive
 $$\varrho(tol)^2= \frac{\V[W(tol)]}{\E[W(tol)]^2} \lesssim  tol^{-\frac{\eta-1}{2r+2}(3 \gamma d-(2r+2))+2},$$ 
 which shows that $\displaystyle \lim_{tol \rightarrow 0}\varrho(tol)=0$, $\forall \eta \in [2, \frac{3 \gamma d+(2r+2)}{3 \gamma d-(2r+2)}[$. In particular,  $\varrho(tol)$ is minimized for $\eta=2$, what finishes the proof.
\end{proof}

We present in the following Section a description of the RMLSG algorithm.

\subsection{Implementation of the RMLSG algorithm} \label{implementationRand}
Here we present an effective implementation of the RMLSG algorithm\\ \\ 
\begin{algorithm}[H]
\caption{Randomized MLSG algorithm}
 \KwData{}
 Choose $\tau_0>\frac{2}{l}, h_0, \overline{\epsilon}_0$,\\
 generate the sequence $L_j=\left\lceil \frac{-1}{\log(2)} \log\left( \frac{1}{h_0}\left(\frac{\overline{\epsilon}_0^2 j^{-1}}{ C(\uu)}\right)^{\frac{1}{2r+2}} \right)    \right\rceil \quad j\geq 1$,\\ 
 compute $p_l^{j}=2^{-l \frac{2r+2+\gamma d}{2}}\left(\sum_{k=0}^{L_j}2^{-k \frac{2r+2+\gamma d}{2}}\right)^{-1} \quad j\geq 1, \quad l=0, \dots, L_j.$\\
\textbf{initialize } $u=0$\;
 \For{$j\geq 1$}{
generate one iid realization of the random field $a_{j}=a(\cdot, \omega_{j})$, $j\geq 1$.\\
sample $l_j \sim \{ p_l^j \}_{l=0}^{L_j}$ on $\{ 0, \dots, L_j \}$\\
 solve primal problem by FE on mesh $h_{l_j-1}$ and realization\\ $a_{j} \rightarrow y^{h_{l_j-1}}(a_{j},u)$ \\
 solve dual problem by FE on mesh $h_{l_j-1}$ and realization\\ $a_{j}\rightarrow p^{h_{l_j-1}}(a_{j},y^{h_{l_j-1}})$\\
 solve primal problem by FE on mesh $h_{l_j}$ and realization\\ $a_{j}\rightarrow y^{h_{l_j}}(a_{j},u)$ \\
 solve dual problem by FE on mesh $h_{l_j}$ and realization\\ $a_{j}\rightarrow p^{h_{l_j}}(a_{j},y^{h_{l_j}})$\\
$\widehat{\nabla J} \leftarrow \beta u+ \frac{1}{p_{l_j}^j}\left( p^{h_{l_j}}(a_{j},u)- p^{h_{{l_j}-1}}(a_{j},u)\right)$\\\
 $u \leftarrow u-\frac{\tau_0}{j} \widehat{\nabla J}$\\
 }
\label{algMLSGRand}
\end{algorithm}

Algorithm \ref{algMLSGRand} requires estimating the constant $C(\uu)$ which can be done in the same way as proposed in Section \ref{implementation}. Notice that overall this randomized version of the MLSG algorithm has less parameters to tune than the non-randomized one.

\section{Numerical results}
\label{numer}
\subsection{Problem setting}
In this section we verify the assertions on the order of convergence and computational complexity stated in Lemmas \ref{lemma:assumptionconvergence}, \ref{lemma:assumptionconvergenceRand} and Theorem \ref{corr:complexity2}, \ref{corr:complexity2Rand} for the MLSG Algorithm \ref{alg:algMLSG} and the RMLSG Algorithm \ref{algMLSGRand}, respectively. For this purpose, we consider the optimal
control problem \eqref{eqn:mother1} in the domain
$D=(0,1)^2$ with $g=1$ and the following random diffusion coefficient:
\begin{multline}
\label{rf}
a(x_1,x_2,{\xi})=1+ \\ \exp\left(var\left(\xi_1 \cos(1.1 \pi x_1)+\xi_2 \cos(1.2 \pi x_1)+\xi_3 \sin(1.3 \pi x_2)+ \xi_4 \sin(1.4 \pi x_2)\right)\right),
\end{multline}
with $(x_1,x_2) \in D$, $var=\exp(-1.125)$ and $\xi = (\xi_1,\dots, \xi_4)$ with
$ \xi_i \overset{iid}{\sim} \mathcal{U}([-1,1])$ (this test case is taken from \cite{LeeGunzburger}). We have chosen $\beta=10^{-4}$ as the price of energy (regularization parameter) in the objective functional, and $z_d(x,y)=\sin(\pi x) \sin( \pi y)$. For the FE approximation,
we have considered a structured triangular grid of mesh size $h$ where each side of the domain $D$
is divided into $1/h$ sub-intervals and used piece-wise linear finite elements (i.e. $r=1$). All calculations have been performed using the FE library Freefem++ \cite{MR3043640}.

\subsection{Reference solution} In order to compute one reference solution $u_{ref}$, we used a tensorized Gauss-Legendre quadrature formula with $q=5$ Gauss-Legendre knots in each of the 4 random variables $\xi_1, \dots, \xi_4$ in \eqref{rf}, hence $n=5^4$ knots in total, in order to approximate the expectation $\E$ in the objective functional. We discretized the OCP \eqref{eqn:mother1}, using a FE method with $\mathbb{P}_1$ elements (i.e. $r=1$), over a regular triangulation of the domain $D$, with a discretization parameter $h=2^{-7}$. To compute one optimal control, we used a full gradient strategy, (requiring at each iteration to solve $2 \times 5^4$ discretized PDE) up to $20$ iterations, using an adaptive (optimal in our quadratic setting) step-size. We reached a final gradient norm of $\|\nabla \widehat{J} (u_{20})\|=6.54276e-12$ and a final difference between two consecutive controls of $\| u_{20}-u_{19} \|=5.05678e-09$.

\subsection{MLSG algorithm}

In order to assess the convergence rate of the MLSG Algorithm \ref{alg:algMLSG} and its computational complexity, we run 10 independent realizations of the MLSG algorithm, up to  120 iterations, using a step size $\tau_j=\tau_0/(j+10)$, and the following parameters:
 $\tau_0=2/\beta$, $h_0=2^{-3}$, $\epsilon_0=\sqrt{C(\uu)h_0^{2r+2}}$, $\sigma_0=\sqrt{\left( 2 \tau_0+\frac{2}{l} \right) \frac{\epsilon_0^2}{2 \tau_0}}$, $\eta=\frac{2(2r+2)-\gamma d}{2r+2-\gamma d}$, $r=1$, $d=2$, $\gamma=1$, $l=2 \beta$ (see Lemma \ref{lemma:SC}).
The constant $C(\uu)$ has been estimated in \cite{MartinSAGA} for the same test case. Here we have taken $C(\uu)=0.5$. These parameters have been used in Algorithm \ref{alg:algMLSG} to determine the levels $L_j$ and samples per level $N_{j,l}$, at each iteration. We report in Table \ref{tab:levelrep} the levels $L_j$ and the corresponding mesh sizes over the iterations 
\begin{table*}\centering
\begin{tabular}{@{}r|c|c|c|c@{}}\toprule
$j: iteration$ & $\{ 0, \dots, 1 \}$ &$\{ 2, \dots, 13 \}$ &$\{ 14, \dots, 62 \}$ &$\{ 63, \dots, 100 \}$ \\
\hline \hline
$L_j$ & 0 & 1 & 2 & 3 \\
$h_{L_j}$ & $2^{-3}$ & $2^{-4}$ & $2^{-5}$ & $2^{-6}$\\
\bottomrule
\end{tabular} 
\caption{One example of level refinement over the iterations for Algorithm \ref{alg:algMLSG}}\label{tab:levelrep}
\end{table*}
In Figure \ref{isocurves1}, we plot the mean error on the control, $\E[\| u_j-u_{ref} \|]$, averaged over the 10 repetitions of the MLSG procedure, versus the iteration counter in $\log$ scale. We verify a slope of $-1.09$, which is consistent with the result $MSE=O(j^{1-\eta})$ stated in Lemma \ref{lemma:assumptionconvergence} with $\eta=3$.
\begin{figure}[!ht]
\centering
\includegraphics[width = 0.9\textwidth]{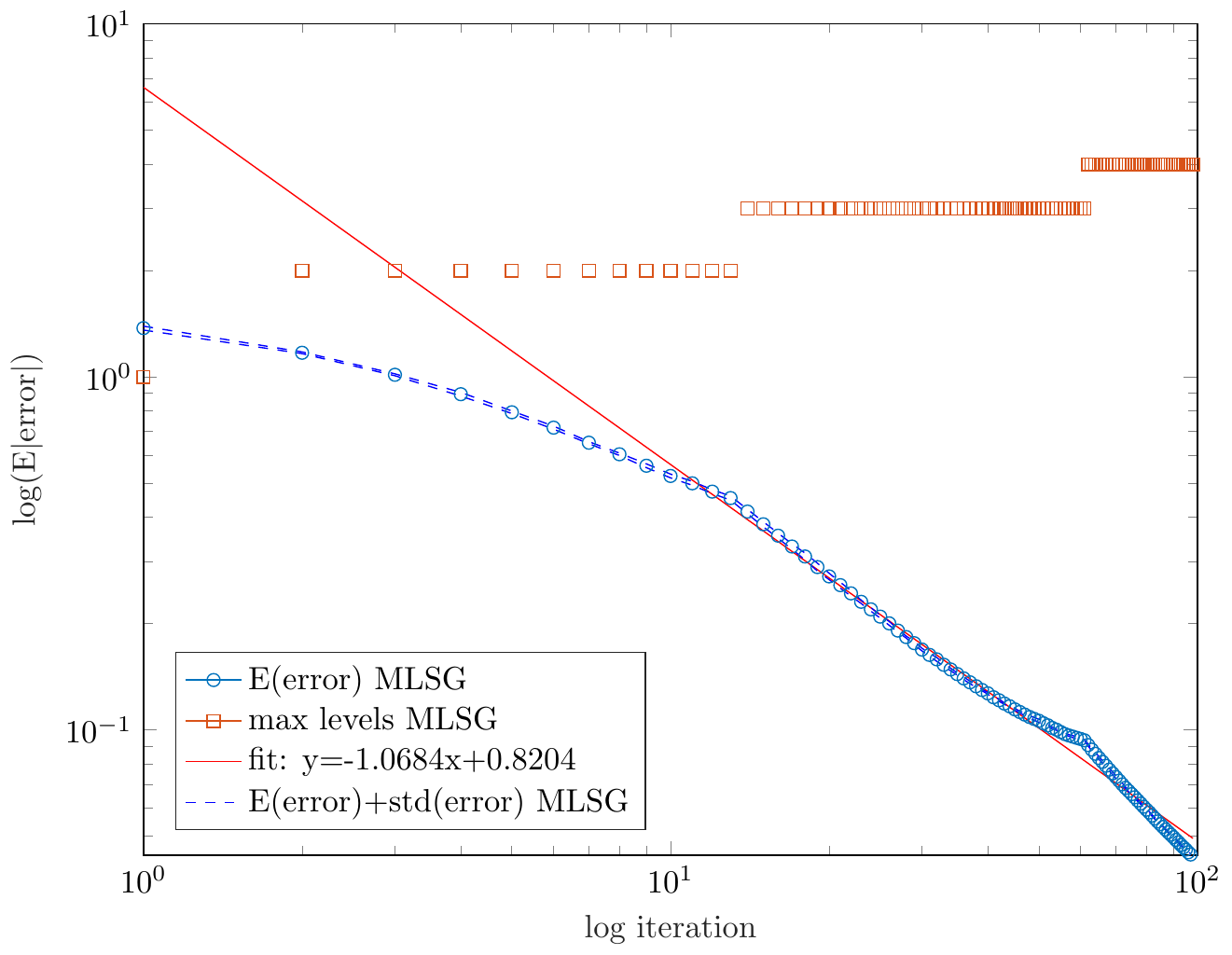}
\caption{Mean error vs iteration counter for the MLSG Algorithm \ref{alg:algMLSG}.}
\label{isocurves1}
\end{figure}
Figure \ref{isocurves2} shows the estimated mean error, averaged over the 10 repetitions, versus the computational cost model $W_j=\sum_{i=0}^{j}\sum_{l=0}^{L_i}2^{l \gamma d}N_{l,i}$, which confirms the complexity result of Theorem \ref{corr:complexity2}.
\begin{figure}[!ht]
\centering
\includegraphics[width = 0.9\textwidth]{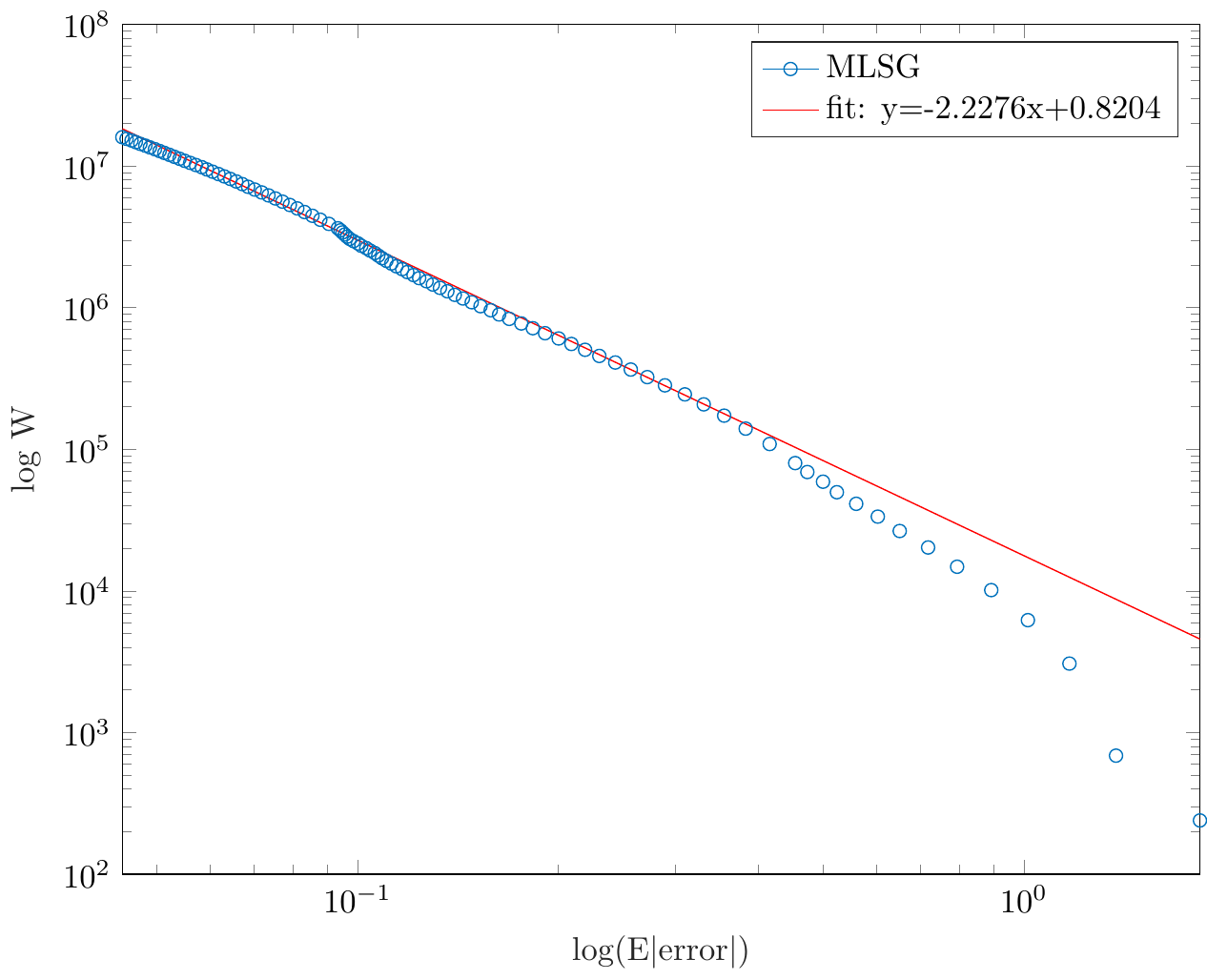}
\caption{Cost $W$ versus mean error, using the MLSG Algorithm \ref{alg:algMLSG}.}
\label{isocurves2}
\end{figure}

\subsection{RMLSG algorithm}
Using the randomized version of the MLSG Algorithm \ref{algMLSGRand}, we assess the convergence rate \eqref{rec2tRand} averaging over 20 independent realizations, of the RMLSG algorithm, up to iteration $j=10000$. The problem setting is the same as above and we have used the same parameters $r=1$. $d=2$, $\gamma=1$, $C(\uu)=0.5$, $\tau_j=\tau_0/(j+10)$ with $\tau_0=2/\beta$, $\overline{\epsilon}_0=\sqrt{C(\uu) h_0^{2r+2}}$, $\eta=2$.
We use the the \emph{optimal} probability mass function:
$$p_l^{j}=2^{-l \frac{2r+2+\gamma d}{2}}\left(\sum_{k=0}^{L_j}2^{-k \frac{2r+2+\gamma d}{2}}\right)^{-1} \quad j\geq 1, \quad l=0, \dots, L_j.$$ 
In Figure \ref{isocurves3}, we plot the mean error versus the iteration counter in $\log$-scale and observe a rate $-1/2$ which is consistent with the result in Lemma \ref{lemma:assumptionconvergenceRand} with $\eta=2$.

\begin{figure}[!ht]
\centering
\includegraphics[width = 0.9\textwidth]{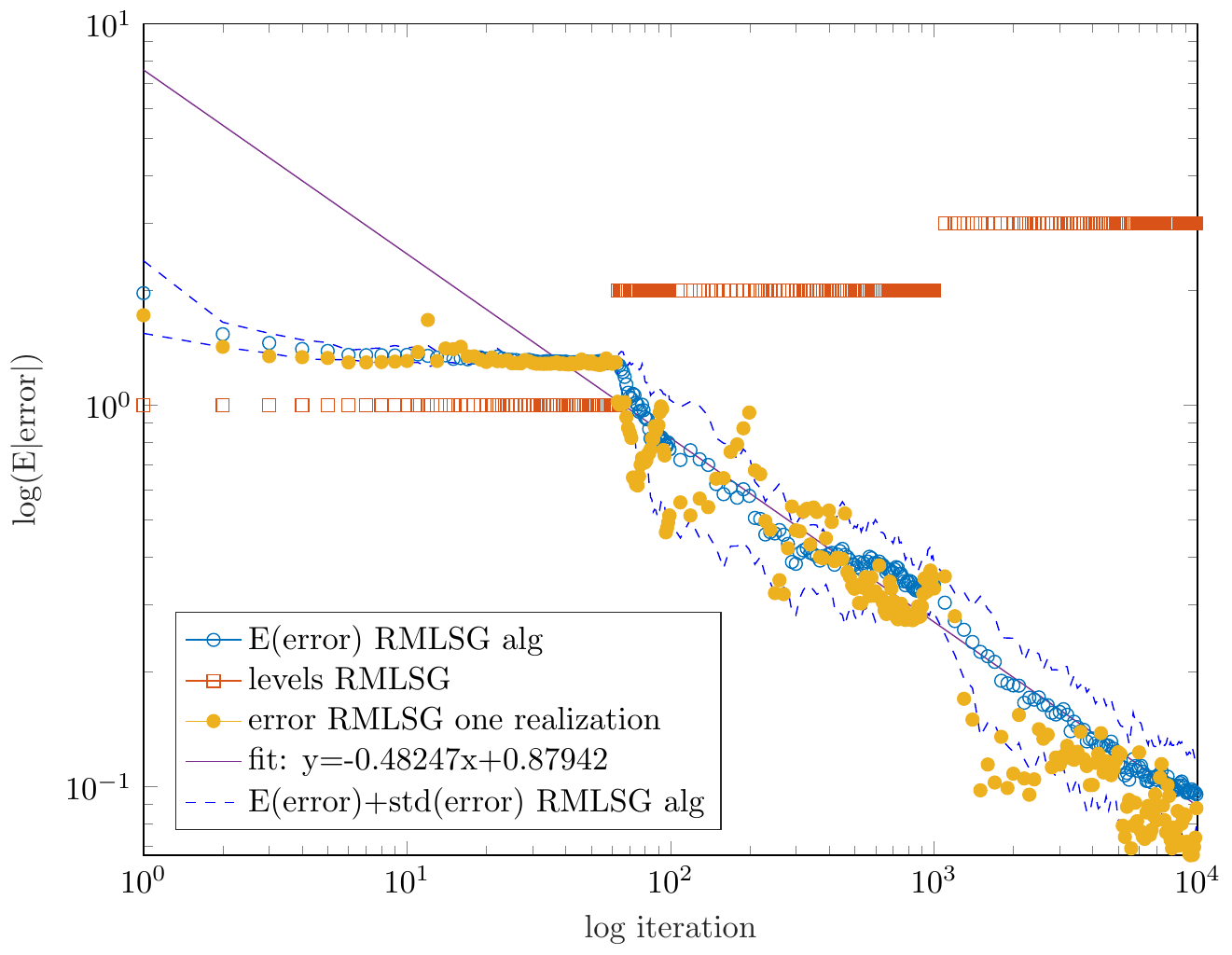}
\caption{Mean error vs iteration counter for the RMLSG Algorithm \ref{algMLSGRand}.}
\label{isocurves3}
\end{figure}
\begin{figure}[!ht]
\centering
\includegraphics[width = 0.9\textwidth]{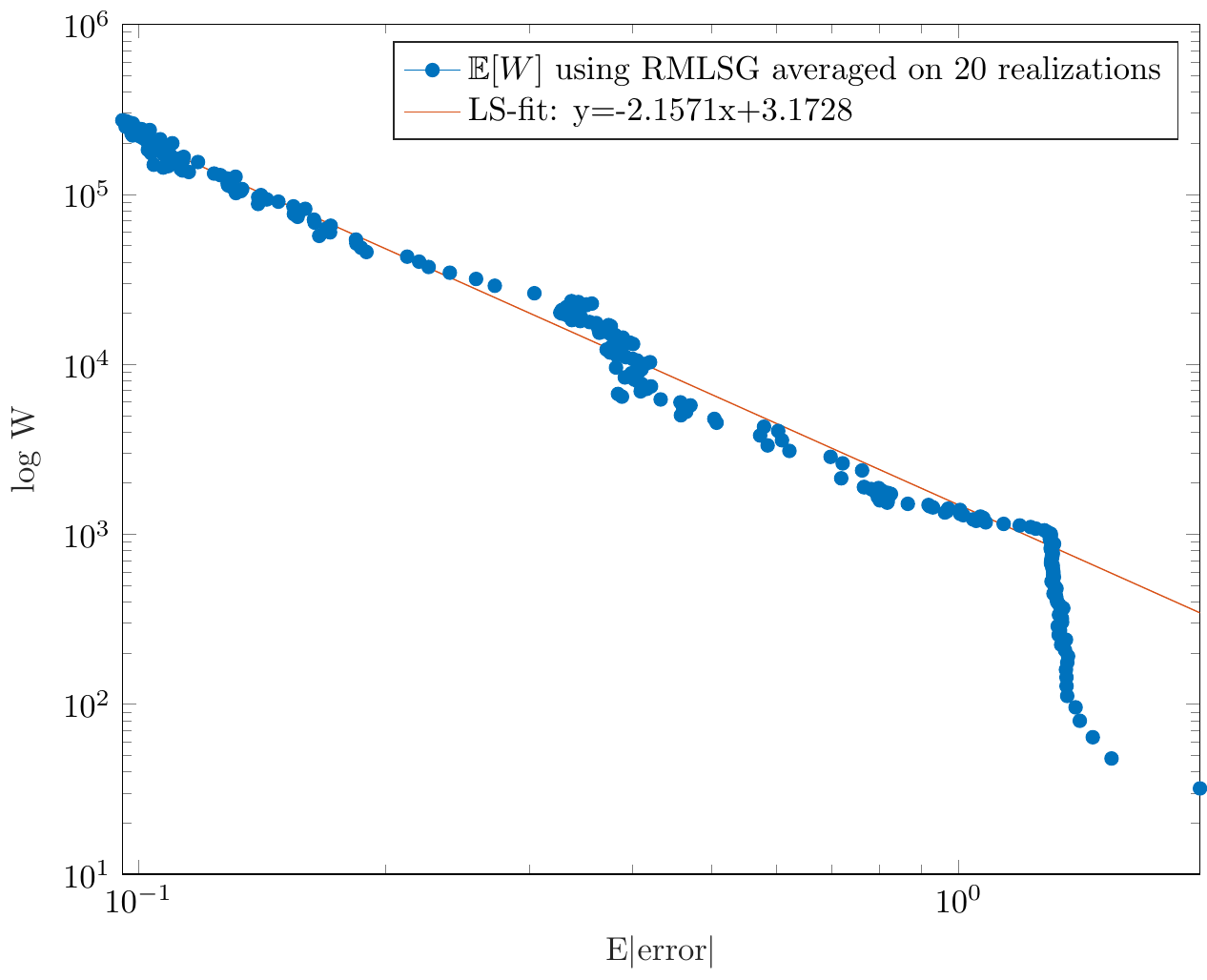}
\caption{Expected cost $\E[W]$ vs mean error for the randomized MLSG Algorithm \ref{algMLSGRand}.}
\label{isocurves4}
\end{figure}
Figure \ref{isocurves4} shows the \emph{expected} computational cost model $\E[W_j]=\sum_{i=1}^{j}\sum_{k=0}^{L_i}2^{k \gamma d} p_k^{i}$ versus the actual mean error averaged over the 20 repetitions and verify a slope of -2 consistent with the complexity result in Theorem \ref{corr:complexity2Rand} (with $\eta=2$). The discontinuities in the expected computational cost in Figure \ref{isocurves4} are due to the fact that the expected cost is not an increasing function of the iteration counter. Specifically, as the probability mass function $\{ p_l^j \}$ is normalized by its sum $\sum_{k'=0}^{L_j}2^{-k' \frac{2r+2+\gamma d}{2}}$, when we reach iteration $j$ where the maximum level $L_j$ is increased the  by 1, we observe slight lower expected cost, i.e. $\E[W(\Er)]=\sum_{k=0}^{L_j}2^{k \gamma d} p_k^{j}=\sum_{k=0}^{L_j} 2^{-k \frac{2r+2-\gamma d}{2}}\left(\sum_{k'=0}^{L_j}2^{-k' \frac{2r+2+\gamma d}{2}}\right)^{-1}$ is not monotonic in $j$.


\section{Conclusions} \label{ccl}
In this work, we presented a modified version of the Stochastic Gradient algorithm, in order to solve numerically a PDE-constrained OCP, with uncertain coefficients. The usual Robbins-Monro approach, involving a single realization estimator of the gradient is replaced by either a MLMC estimator, with increasing cost w.r.t. the iteration counter, or a randomized version of the MLMC estimator, where only one difference term of the full MLMC estimator is computed at each iteration, on a randomly drawn level, according to a probability mass function, set a priori. We have shown that both algorithms, when properly tuned, achieve the optimal complexity $W \lesssim tol^{-2}$ in certain cases. 
These complexity results are assessed in the numerical Section. In practice, many constants have to be tuned beforehand in these two MLSG algorithms. Our preference goes to the randomized version RMLSG as it presents fewer parameters to tune. 


\bibliographystyle{siamplain}
\bibliography{bib}
\end{document}